\documentclass[11pt,reqno]{amsart}

\newtheorem{proposition}{Proposition}[section]
\newtheorem{lemma}[proposition]{Lemma}
\newtheorem{corollary}[proposition]{Corollary}
\newtheorem{theorem}[proposition]{Theorem}

\theoremstyle{definition}
\newtheorem{definition}[proposition]{Definition}
\newtheorem{example}[proposition]{Example}
\newtheorem{examples}[proposition]{Examples}
\newtheorem{remark}[proposition]{Remark}

\newcommand{\thlabel}[1]{\label{th:#1}}
\newcommand{\thref}[1]{Theorem~\ref{th:#1}}
\newcommand{\selabel}[1]{\label{se:#1}}
\newcommand{\seref}[1]{Section~\ref{se:#1}}
\newcommand{\lelabel}[1]{\label{le:#1}}
\newcommand{\leref}[1]{Lemma~\ref{le:#1}}
\newcommand{\prlabel}[1]{\label{pr:#1}}
\newcommand{\prref}[1]{Proposition~\ref{pr:#1}}
\newcommand{\colabel}[1]{\label{co:#1}}
\newcommand{\coref}[1]{Corollary~\ref{co:#1}}
\newcommand{\relabel}[1]{\label{re:#1}}
\newcommand{\reref}[1]{Remark~\ref{re:#1}}
\newcommand{\exlabel}[1]{\label{ex:#1}}
\newcommand{\exref}[1]{Example~\ref{ex:#1}}
\newcommand{\delabel}[1]{\label{de:#1}}
\newcommand{\deref}[1]{Definition~\ref{de:#1}}
\newcommand{\eqlabel}[1]{\label{eq:#1}}
\newcommand{\equref}[1]{(\ref{eq:#1})}

\newcommand{\Hom}{{\rm Hom}}
\newcommand{\End}{{\rm End}}

\newcommand{\Aut}{{\rm Aut}\,}

\def\lan{\langle}
\def\ran{\rangle}
\def\ot{\otimes}

\def\ZZ{{\mathbb Z}}

\newcommand{\Cc}{\mathcal{C}}

\def\*C{{}^*\hspace*{-1pt}{\Cc}}
\def\text#1{{\rm {\rm #1}}}

\input xy
\xyoption {all} \CompileMatrices

\usepackage{amssymb}
\usepackage{color,amssymb,graphicx,amscd}

\begin{document}

\title[Classifying bicrossed products of Hopf algebras]
{Classifying bicrossed products of Hopf algebras}

\author{A. L. Agore}
\address{Faculty of Engineering, Vrije Universiteit Brussel, Pleinlaan 2, B-1050 Brussels, Belgium}
\email{ana.agore@vub.ac.be and ana.agore@gmail.com}

\author{C.G. Bontea}
\address{Faculty of Mathematics and Computer Science, University of Bucharest, Str.
Academiei 14, RO-010014 Bucharest 1, Romania}
\email{costel.bontea@gmail.com}

\author{G. Militaru}
\address{Faculty of Mathematics and Computer Science, University of Bucharest, Str.
Academiei 14, RO-010014 Bucharest 1, Romania}
\email{gigel.militaru@fmi.unibuc.ro and gigel.militaru@gmail.com}

\thanks{A.L. Agore is research fellow ''aspirant'' of FWO-Vlaanderen.
This work was supported by a grant of the Romanian National
Authority for Scientific Research, CNCS-UEFISCDI, grant no.
88/05.10.2011.}



\begin{abstract}
Let $A$ and $H$ be two Hopf algebras. We shall classify up to an
isomorphism that stabilizes $A$ all Hopf algebras $E$ that
factorize through $A$ and $H$ by a cohomological type object
${\mathcal H}^{2} (A, H)$. Equivalently, we classify up to a left
$A$-linear Hopf algebra isomorphism, the set of all bicrossed
products $A \bowtie H$ associated to all possible matched pairs of
Hopf algebras $(A, H, \triangleleft, \triangleright)$ that can be
defined between $A$ and $H$. In the construction of ${\mathcal
H}^{2} (A, H)$ the key role is played by special elements of
$CoZ^{1} (H, A) \times \Aut_{\rm CoAlg}^1 (H)$, where $CoZ^{1} (H,
A)$ is the group of unitary cocentral maps and $\Aut_{\rm CoAlg}^1
(H)$ is the group of unitary automorphisms of the coalgebra $H$.
Among several applications and examples, all bicrossed products
$H_4 \bowtie k[C_n]$ are described by generators and relations and
classified: they are quantum groups at roots of unity $H_{4n, \,
\omega}$ which are classified by pure arithmetic properties of the
ring $\mathbb{Z}_n$. The Dirichlet's theorem on primes is used to
count the number of types of isomorphisms of this family of
$4n$-dimensional quantum groups. As a consequence of our approach
the group $\Aut_{\rm Hopf}(H_{4n, \, \omega})$ of Hopf algebra
automorphisms is fully described.
\end{abstract}

\maketitle

\section*{Introduction}
The aim of this paper is to prove that there exists a very rich
theory behind the so-called bicrossed product (double cross
product in Majid's terminology) of two objects arising from the
factorization problem. This theory deserves to be developed
further mainly because of its major impact in at least three
different problems: the classification of objects of a given
dimension, the development of a general descent type theory for a
given extension $ A \subseteq E$ recently introduced in
\cite{am12} and the development of some new types of cohomologies
that will be the key players for both problems. All results
presented below provide an answer at the level of Hopf algebras
for the classification problem and offer an argument for the role
that bicrossed products can play. In particular, we describe a
tempting way of approaching the classification problem for finite
quantum groups. The pioneers of this approach, at the level of
groups, were Douglas \cite{Douglas}, R\'{e}dei \cite{Redei} and
Cohn \cite{Cohn} related to the problem of classifying all groups
that factorize through two cyclic groups.

In order to maintain a general frame for our discussion,
considering that bicrossed products of two objects were introduced
and studied in various areas of mathematics, we will consider
$\Cc$ a category whose objects are sets endowed with various
algebraic, topological or differential structures. To illustrate,
we can think of $\Cc$ as the category of groups, groupoids or
quantum groupoids, algebras, Hopf algebras, local compact groups
or local compact quantum groups, Lie groups, Lie algebras and so
on. Let $A$ and $H$ be two given objects of $\Cc$. We say that an
object $E \in \Cc$ \emph{factorizes} through $A$ and $H$ if $E$
can be written as a 'product' of $A$ and $H$, where $A$ and $H$
are subobjects of $E$ having minimal intersection. Here, the
'product' depends on the nature of the category. For instance, if
$\Cc$ is the category of groups, then a group $E$ factorizes
through two subgroups $A$ and $H$ if $E = AH$ and $A \cap H =
\{1\}$. This is called in group theory an exact factorization of
$E$ and can be restated equivalently as the fact that the
multiplication map $ A \times H \to E$, $(a, h) \mapsto ah$ is
bijective. The last assertion is also taken as a definition of
factorization for other categories like: algebras \cite{cap}, Hopf
algebras \cite{majid}, Lie groups or Lie algebras \cite{majid3},
\cite{Kro}, \cite{Mic}, locally compact quantum groups \cite{VV}
and so on. The following natural question arises:

\textbf{The factorization problem.} \textit{Let $A$ and $H$ be two
given objects of $\Cc$. Describe and classify up to an isomorphism
all objects of $\Cc$ that factorize through $A$ and $H$.}

We shall provide a short historical background of this problem.
First we consider the case in which the category $\Cc = {\mathcal
Ab}$, where ${\mathcal Ab}$ is the category of abelian groups or,
more specifically, the category of left $R$-modules over a ring
$R$. In this case the factorization problem becomes trivial: an
abelian group $E$ factorizes through two subgroups $A$ and $H$ if
and only if $E \cong A \oplus H$, that is $E$ is the coproduct of
$A$ and $H$. Now we take a step forward and consider $\Cc =
{\mathcal Gr }$, the category of groups. Here things change
radically: the factorization problem becomes very difficult and it
was formulated by O. Ore \cite{Ore} but its roots are much older
and descend to E. Maillet's 1900 paper \cite{Maillet}. It can be
seen as the dual of the more famous \textit{extension problem} of
O. L. H\"{o}lder. Moreover, little progress has been made on this
problem so far. For instance, in the case of two cyclic groups $A$
and $H$, not both finite, the problem was opened by R\'{e}dei in
\cite{Redei} and closed by Cohn in \cite{Cohn}, without the
classification part. If $A$ and $H$ are both finite cyclic groups
the problem is more difficult and seems to be still an open
question, even though J. Douglas \cite{Douglas} has devoted four
papers to the subject. In \cite{dar} all groups that factorize
through an alternating group or a symmetric group are completely
described. One of the famous results in this direction remains
Ito's theorem \cite{Ito} proved in the 50's: any product of two
abelian groups is a meta-abelian group. An important step in
dealing with the factorization problem was the construction of the
bicrossed product $A \bowtie H $ associated to a matched pair $(A,
H, \triangleleft, \triangleright )$ of groups, given by Takeuchi
\cite{Takeuchi}. A group $E$ factorizes through two subgroups $A$
and $H$ if and only if there exists a matched pair of groups $(A,
H, \triangleleft, \triangleright )$ such that $E \cong A \bowtie
H$. Thus, at the level of groups, the factorization problem can be
restated in a purely computational manner: \emph{Let $A$ and $H$
be two given groups. Describe the set of all matched pairs $(A, H,
\triangleleft, \triangleright )$ and classify up to an isomorphism
all bicrossed products $A \bowtie H$.}

This is a strategy that has to be followed for the factorization
problem in any category $\Cc$. In fact, the bicrossed product
construction at the level of groups served as a model for similar
constructions in other fields of mathematics. The first step was
made by Majid in \cite[Proposition 3.12]{majid} where a twisted
version of the bicrossed product of two Hopf algebras associated
to a matched pair was introduced, under the name of double cross
product. Its purpose was not to solve the factorization problem,
but to give an elegant description for the Drinfel'd double of a
finite dimensional Hopf algebra. Then the construction was
performed for Lie Algebras \cite{majid3}, \cite{Mic} and Lie
groups \cite{majid3}, \cite{Kro}, algebras \cite{cap} or more
generally in \cite{cirio}, coalgebras \cite{CIMZ}, groupoids
\cite{AA}, \cite{majard}, locally compact groups \cite{baaj} or
locally compact quantum groups \cite{VV}, and so on. At the level
of algebras the bicrossed product of two algebras is usually
called the \emph{twisted tensor product} algebra and its
construction plays an important role in noncommutative geometry in
the sense of Brezinski and Majid (see \cite{bm98}, \cite{bm200},
\cite{brz01}).

To conclude, if we are only looking for the description part of
the factorization problem, formulated in an arbitrary category
$\Cc$, the following general principle should work: an object
$E\in \Cc$ factorizes through $A$ and $H$ if and only if $E \cong
A \bowtie H$, where $A \bowtie H$ is a 'bicrossed product' in the
category $\Cc$ associated to a 'matched pair' between the objects
$A$ and $H$. The classification part of the factorization problem
is now clear: it consists of classifying the bicrossed products $A
\bowtie H$ associated to all matched pairs between $A$ and $H$.
This is the strategy that we follow for the category of Hopf
algebras. For other categories, the steps taken in this direction
are still shy, including the group case as well as the algebra
case. The problem of classifying bicrossed products of two
algebras started with \cite[Examples 2.11]{CIMZ} where all
bicrossed products between two group algebras of dimension two are
completely described and classified. For recent results related to
the classification of bicrossed product of two algebras we refer
to \cite{Pena}, \cite{Pena2}, \cite{gucci}.

Looking at the classification part of the factorization problem
for Hopf algebras we emphasize that the theory of bicrossed
products can play an important role, not exploited until now, in
the classification of finite quantum groups. The problem of
classifying up to isomorphism all Hopf algebras of a given
dimension is one of the central themes in Hopf algebra theory.
There are complete classifications for Hopf algebras of given
small dimensions: see for instance, \cite{AD}, \cite{AS},
\cite{beat}, \cite{HNG}, \cite{mas2}, \cite{Ng}, \cite{natale},
\cite{Zhu} and their list of references. However, there are no
general methods of tackling this problem. Following H\"{o}lder's
model from group extensions, the only method for classifying a
certain class (pointed, semisimple, etc) of Hopf algebras of a
given dimension relies on the theory of Hopf algebra extensions
\cite{AD}, \cite{mas} and the Radford biproduct. Unfortunately,
for Hopf algebras the method has its limitations and is not an
exhaustive one. We shall highlight a new way of classifying Hopf
algebras of a given dimension using bicrossed products instead of
extension theory for Hopf algebras.

The paper is organized as follows. In \seref{prel} we recall the
construction of the bicrossed product associated to a matched pair
of Hopf algebras $(A, H, \triangleleft, \triangleright)$. Majid's
\thref{carMaj} proves that the factorization problem for Hopf
algebras can be restated in a computational manner exactly as in
the group case: \textit{Let $A$ and $H$ be two given Hopf
algebras. Describe the set of all matched pairs $(A, H,
\triangleright, \triangleleft)$ and classify up to an isomorphism
all bicrossed products $A \bowtie \, H$.}

\seref{sectth} is devoted to proving some purely technical results
which will be intensively used throughout the paper.
\thref{toatemorf} describes completely the set of all morphisms of
Hopf algebras between two arbitrary bicrossed products. In
particular, the set of all Hopf algebra maps between two
semi-direct (or smash) products of Hopf algebras is described in
\coref{endosmh}. This result gives the first application: the
parametrization of all Hopf algebra morphisms between two
Drinfel'd doubles associated to two finite groups $G$ and $H$ is
given in \coref{endoDH}. In particular, if $G = H$, the space
$\End_{\rm Hopf} \bigl(D(k[G])\bigl)$ of all Hopf algebra
endomorphisms is fully described.

\seref{casbic} deals with the classification part of the
factorization problem. Let $A$ and $H$ be two given Hopf algebras.
\thref{clasth} is the classification theorem for bicrossed
products: all Hopf algebras $E$ that factorize through $A$ and $H$
are classified up to an isomorphism that stabilizes $A$ by a
cohomological type object ${\mathcal H}^{2} (A, H)$ in the
construction of which the key role is played by pairs $(r, v) \in
CoZ^{1} (H, A) \times \Aut_{\rm CoAlg}^1 (H)$, consisting of a
unitary cocentral map $r : H\to A$ and a unitary automorphism of
coalgebras $v: H\to H$ related by a certain compatibility
condition. The classification of bicrossed products up to an
isomorphism that stabilizes one of the terms has at least two
strong motivations: the first one is the cohomological point of
view which descends from the classification theory of Holder's
group extensions \cite[Theorem 7.34]{R} and the second one is the
problem of describing and classifying the factorization $A$-forms
of a Hopf algebra extension from descent theory \cite{am12}. We
indicate only two of the several applications of \thref{clasth}:
\coref{douadh} gives necessary and sufficient conditions for two
generalized quantum doubles $D_{\lambda} (A, H)$ and $D_{\lambda'}
(A, H')$ to be isomorphic as Hopf algebras and left $A$-modules.
\coref{SchZ} is a Schur-Zassenhaus type theorem for the bicrossed
product of Hopf algebras: it provides necessary and sufficient
conditions for a bicrossed product $A \bowtie H$ to be isomorphic
to a semi-direct product $A\#' H'$ of Hopf algebras.

In \seref{secexemple} we provide some explicit examples for the
factorization problem: for two given Hopf algebras $A$ and $H$ we
describe by generators and relations and classify up to an
isomorphism all Hopf algebras $E$ that factorize through $A$ and
$H$. We go through the following three steps: first of all, we
compute the set of all matched pairs between $A$ and $H$. This is
the computational part of our schedule and can not be avoided.
Then we describe by generators and relations the bicrossed
products $A \bowtie H$ associated to all these matched pairs.
Finally, using \thref{toatemorf}, we classify up to an isomorphism
these bicrossed products $A \bowtie H$. As an application, the
group $\Aut _{\rm Hopf} (A \bowtie H)$ of all Hopf algebra
automorphisms of a given bicrossed product is computed.

Let $k$ be a field of characteristic $\neq 2$ and $H_4$ the
Sweedler's four dimensional Hopf algebra. For a positive integer
$n$, let $C_n$ be the cyclic group of order $n$ generated by $c$
and $U_n (k) = \{ \omega \in k \, | \, \omega^n = 1 \}$ the cyclic
group of $n$-th roots of unity in $k$ of order $\nu (n) = |U_n
(k)|$. The group $U_n (k)$ depends heavily on the base field $k$.
\prref{mapex1} proves that $U_n (k)$ parameterizes the set of all
matched pairs $(H_4, k[C_n], \triangleleft, \triangleright)$, i.e.
there exists a bijective correspondence between the set of all
matched pairs $(H_4, k[C_n], \triangleleft, \triangleright)$ and
the group $U_n (k)$. \coref{primaclasi} shows that a Hopf algebra
$E$ factorizes through $H_4$ and $k[C_n]$ if and only if $E \cong
H_{4n, \, \omega}$, for some $\omega \in U_n(k)$; this quantum
group $H_{4n, \, \omega}$ at root of unity is described explicitly
by generators and relations. It is a non-commutative
non-cocommutative, pointed and non-semisimple $4n$-dimensional
Hopf algebra. \thref{2.2} and \thref{clasnoudir} describe
precisely the number of types of isomorphisms of this family of
Hopf algebras $\{ H_{4n, \, \omega} \, | \, \omega \in U_n(k) \}$.
The beauty of this classification result is given by Dirichlet's
theorem on primes in an arithmetical progression which was used in
a key step in proving \thref{clasnoudir}. Let $\nu(n) =
p_1^{\alpha_1} \cdots p_r^{\alpha_r}$ be the prime decomposition
of $\nu(n)$. If $\nu(n)$ is odd, then the number of types of such
Hopf algebras is $(\alpha_1 + 1)(\alpha_2 + 1) \cdots (\alpha_r +
1)$. On the other hand, if $\nu(n) = 2^{\alpha_1} p_2^{\alpha_2}
\cdots p_r^{\alpha_r}$ is even, then the number of types of such
Hopf algebras is $\alpha_1(\alpha_2 + 1) \cdots (\alpha_r + 1)$.
As an application, the infinite abelian group $\Aut_{\rm
Hopf}(H_{4n, \, \omega})$ of Hopf algebra automorphisms of $H_{4n,
\, \omega}$ is described in \coref{auto}.

The results proven in this paper at the level of Hopf algebras can
serve as a model for obtaining similar results for all the fields
of mathematics where the bicrossed product is constructed.

\section{Preliminaries}\selabel{prel}
Throughout this paper, $k$ will be a field. For a positive integer
$n$ we denote by $U_n (k) = \{ \omega \in k \, | \, \omega^n = 1
\}$ the cyclic group of $n$-th roots of unity in $k$, and by $\nu
(n) = |U_n (k)|$ the order of the group $U_n (k)$. Of course, $\nu
(n)$ is a divisor of $n$; if $\nu (n) = n$, then any generator of
$U_n (k)$ is called a primitive $n$-th root of unity \cite{hung}.
The group $U_n (k)$ depends heavily on the base field $k$. For
instance, $U_{p^t} (k) = \{1\}$, for any positive integer $t$, if
$k$ is a field of characteristic $p > 0$. If ${\rm Char}(k) = p$
and $p | n$ then there is no primitive $n$-th root of unity in $k$
\cite[Pag. 295]{hung}. Furthermore, if $k$ is a finite field with
$p^k$ elements, then $\nu (n) = {\rm gcd} \, (n, p^k -1)$. If
${\rm Char}(k) \nmid n$ and $k$ is algebraically closed then $\nu
(n) = n$.

Unless specified otherwise, all algebras, coalgebras, bialgebras,
Hopf algebras, tensor products and homomorphisms are over $k$. For
a coalgebra $C$, we use Sweedler's $\Sigma$-notation: $\Delta(c) =
c_{(1)}\ot c_{(2)}$, $(I\ot\Delta)\Delta(c) = c_{(1)}\ot
c_{(2)}\ot c_{(3)}$, etc (summation understood). Let $A$ and $H$
be two Hopf algebras. $H$ is called a right $A$-module coalgebra
if $H$ is a coalgebra in the monoidal category ${\mathcal M}_A $
of right $A$-modules, i.e. there exists $\triangleleft : H \otimes
A \rightarrow H$ a morphism of coalgebras such that $(H,
\triangleleft) $ is a right $A$-module. A morphism between two
right $A$-module coalgebras $(H, \triangleleft)$ and $(H',
\triangleleft')$ is a morphism of coalgebras $\psi: H \to H'$ that
is also right $A$-linear. Furthermore, $\psi$ is called unitary if
$\psi (1_H) = 1_{H'}$. Similarly, $A$ is a left $H$-module
coalgebra if $A$ is a coalgebra in the monoidal category of left
$H$-modules, that is there exists $\triangleright : H \ot A \to A$
a morphism of coalgebras such that $(A, \triangleright)$ is also a
left $H$-module. For further computations, the fact that two
$k$-linear maps $\triangleleft : H \otimes A \rightarrow H$,
$\triangleright: H \otimes A \rightarrow A$ are coalgebra maps can
be written explicitly as follows:
\begin{eqnarray}
\Delta_{H}(h \triangleleft a) &{=}& h_{(1)} \triangleleft a_{(1)}
\otimes h_{(2)} \triangleleft a_{(2)}, \,\,\,\,\,
\varepsilon_{A}(h \triangleleft a) =
\varepsilon_{H}(h)\varepsilon_{A}(a)
\eqlabel{6}\\
\Delta_{A}(h \triangleright a) &{=}& h_{(1)} \triangleright
a_{(1)} \otimes h_{(2)} \triangleright a_{(2)}, \,\,\,\,\,
\varepsilon_{A}(h \triangleright a) =
\varepsilon_{H}(h)\varepsilon_{A}(a) \eqlabel{8}
\end{eqnarray}
for all $h \in H$, $a \in A$. The actions $\triangleleft : H
\otimes A \rightarrow H$, $\triangleright: H \otimes A \rightarrow
A$ are called \emph{trivial} if $h \triangleleft a = \varepsilon_A
(a) h$ and respectively $h \triangleright a = \varepsilon_H(h) a$,
for all $a\in A$ and $h\in H$.

\subsection*{Bicrossed product of Hopf algebras}\selabel{1.2}
The bicrossed product of two Hopf algebras was introduced by Majid
in \cite[Proposition 3.12]{majid} under the name of double cross
product. We shall adopt the name of bicrossed product from
\cite[Theorem IX 2.3]{Kassel} that is also used for the similar
construction at the level of groups \cite{Takeuchi}, groupoids
\cite{AA} etc. A \textit{matched pair} of Hopf algebras is a
system $(A, H, \triangleleft, \triangleright)$, where $A$ and $H$
are Hopf algebras, $\triangleleft : H \otimes A \rightarrow H$,
$\triangleright: H \otimes A \rightarrow A$ are coalgebra maps
such that $(A, \triangleright)$ is a left $H$-module coalgebra,
$(H, \triangleleft)$ is a right $A$-module coalgebra and the
following compatibilities hold for any $a$, $b\in A$, $g$, $h\in
H$.
\begin{eqnarray}
h \triangleright1_{A} &{=}& \varepsilon_{H}(h)1_{A}, \quad 1_{H}
\triangleleft a =
\varepsilon_{A}(a)1_{H} \eqlabel{mp1} \\
g \triangleright(ab) &{=}& (g_{(1)} \triangleright a_{(1)}) \bigl
( (g_{(2)}\triangleleft a_{(2)})\triangleright b \bigl)
\eqlabel{mp2} \\
(g h) \triangleleft a &{=}& \bigl( g \triangleleft (h_{(1)}
\triangleright a_{(1)}) \bigl) (h_{(2)} \triangleleft a_{(2)})
\eqlabel{mp3} \\
g_{(1)} \triangleleft a_{(1)} \otimes g_{(2)} \triangleright
a_{(2)} &{=}& g_{(2)} \triangleleft a_{(2)} \otimes g_{(1)}
\triangleright a_{(1)} \eqlabel{mp4}
\end{eqnarray}
Let $(A, H, \triangleleft, \triangleright)$ be a matched pair of
Hopf algebras; the \textit{bicrossed product} $A \bowtie H$ of $A$
with $H$ is the $k$-module $A\ot H$ with the multiplication given
by
\begin{equation}\eqlabel{0010}
(a \bowtie h) \cdot (c \bowtie g):= a (h_{(1)}\triangleright
c_{(1)}) \bowtie (h_{(2)} \triangleleft c_{(2)}) g
\end{equation}
for all $a$, $c\in A$, $h$, $g\in H$, where we denoted $a\ot h$ by
$a\bowtie h$. $A \bowtie H$ is a Hopf algebra with the antipode
\begin{equation}\eqlabel{antipbic}
S_{A \bowtie H} ( a \bowtie h ) = S_H (h_{(2)}) \triangleright S_A
(a_{(2)}) \, \bowtie \, S_H (h_{(1)}) \triangleleft S_A (a_{(1)})
\end{equation}
for all $a\in A$ and $h\in H$.

\begin{examples} \exlabel{exempleban}
1. Let $(A, \triangleright)$ be a left $H$-module coalgebra and
consider $H$ as a right $A$-module coalgebra via the trivial
action, i.e. $h \triangleleft a = \varepsilon_A(a) h$. Then $(A,
H, \triangleleft, \triangleright)$ is a matched pair of Hopf
algebras if and only if $(A, \triangleright)$ is also a left
$H$-module algebra and the following compatibility condition holds
\begin{equation}\eqlabel{smash1}
g_{(1)} \otimes g_{(2)} \triangleright a =  g_{(2)} \otimes
g_{(1)} \triangleright a
\end{equation}
for all $g \in H$ and $a\in A$. In this case, the associated
bicrossed product $A\bowtie H = A\# H$ is the left version of the
semi-direct (smash) product of Hopf algebras as defined by Molnar
\cite{Mo} in the cocommutative case, for which the compatibility
condition \equref{smash1} holds automatically. Thus, $A\# H$ is
the $k$-module $A\ot H$, where the multiplication \equref{0010}
takes the form:
\begin{equation}\eqlabel{smash2}
(a \# h) \cdot (c \# g):= a \, (h_{(1)} \triangleright c )\,\#\,
h_{(2)} g
\end{equation}
for all $a$, $c\in A$, $h$, $g\in H$, where we denoted $a\ot h$ by
$a\# h$. $A \# H$ is a Hopf algebra with the coalgebra structure
given by the tensor product of coalgebras and the antipode
$$
S_{A \# H} ( a \# h ) = S_H (h_{(2)}) \triangleright S_A (a) \, \#
\, S_H (h_{(1)})
$$
for all $a\in A$ and $h\in H$.

Similarly, let $(H, \triangleleft)$ be a right $A$-module
coalgebra and consider $A$ as left $H$-module coalgebra via the
trivial action, i.e $h \triangleright a := \varepsilon_H(h) a$.
Then $(A, H, \triangleleft, \triangleright)$ is a matched pair of
Hopf algebras if and only if $(H, \triangleleft)$ is also a right
$A$-module algebra and
\begin{equation}\eqlabel{smash3}
g \triangleleft a_{(1)} \otimes a_{(2)} = g \triangleleft a_{(2)}
\otimes a_{(1)}
\end{equation}
for all $g\in H$ and $a\in A$. In this case, the associated
bicrossed product $A\bowtie H = A\#^r H$ is the right version of
the smash product of Hopf algebras. Thus, $A\#^r H$ is the
$k$-module $A\ot H$, where the multiplication \equref{0010} takes
the form:
\begin{equation}\eqlabel{smash4}
(a \# h) \cdot (c \# g):= a \, c_{(1)} \,\#\, (h \triangleleft
c_{(2)}) g
\end{equation}
for all $a$, $c\in A$, $h$, $g\in H$, where we denoted $a\ot h$ by
$a\# h$. $A \#^r H$ is a Hopf algebra with the coalgebra structure
given by the tensor product of coalgebras and the antipode
$$
S_{A \#^r H} ( a \# h ) = S_A (a_{(2)}) \, \# \, S_H (h)
\triangleleft S_A (a_{(1)})
$$
for all $a\in A$ and $h\in H$.

2. Let $G$ and $K$ be two groups, $A= k[G]$ and $H = k[K]$ the
group algebras. There exists a bijection between the set of all
matched pairs of Hopf algebras $(k[G], k[K], \triangleleft,
\triangleright)$ and the set of all matched pairs of groups $(G,
K, \tilde{\triangleleft}, \tilde{\triangleright})$ in the sense of
Takeuchi \cite{Takeuchi}. The bijection is given such that there
exists an isomorphism of Hopf algebras $k[G] \bowtie k[K] \cong k
[G \bowtie H]$, where $G \bowtie H$ is the Takeuchi's bicrossed
product of groups (\cite[pg. 207]{Kassel}.

3. The fundamental example of a bicrossed product is the Drinfel'd
double $D(H)$. Let $H$ be a finite dimensional Hopf algebra. Then
we have a matched pair of Hopf algebras $( (H^*)^{\rm cop}, H,
\triangleleft, \triangleright)$, where the actions $\triangleleft$
and $\triangleright$ are defined by:
\begin{equation} \eqlabel{dubluact}
h \triangleleft h^* := \lan h^*, \, S_H^{-1} (h_{(3)}) h_{(1)}
\ran h_{(2)}, \quad h \triangleright h^* := \lan h^*, \, S_H^{-1}
(h_{(2)}) \, ? \, h_{(1)} \ran
\end{equation}
for all $h\in H$ and $h^* \in H^*$ (\cite[Theorem
IX.3.5]{Kassel}). The Drinfel'd double of $H$ is the bicrossed
product associated to this matched pair, i.e. $D(H) = (H^*)^{\rm
cop} \bowtie H$.

4. Majid generalized the construction of the Drinfel'd double to
the infinite dimensional case. Let $A$ and $H$ be two Hopf
algebras and $\lambda: H \ot A \rightarrow k$ a skew pairing. Then
there exists a matched pair of Hopf algebras $(A, H, \triangleleft
= \triangleleft_{\lambda}, \, \triangleright =
\triangleright_{\lambda})$, where the actions $\triangleleft$,
$\triangleright$ arise from $\lambda$ via
\begin{eqnarray}
h \triangleleft a &{=}& h_{(2)} \lambda^{-1}(h_{(1)} , a_{(1)})
\lambda(h_{(3)} , a_{(2)}) = h_{(2)} \lambda \bigl( S( h_{(1)})
h_{(3)} , \, a \bigl ) \eqlabel{dhsk1}\\
h \triangleright a &{=}& a_{(2)} \lambda^{-1}(h_{(1)} , a_{(1)})
\lambda(h_{(2)} , a_{(3)}) = a_{(2)} \lambda \bigl( S(h) , \,
a_{(1)} S(a_{(3)}) \bigl) \eqlabel{dhsk2}
\end{eqnarray}
for all $h\in H$ and $a\in A$ (\cite[Example 7.2.6]{majid}). The
corresponding bicrossed product $A \bowtie_{\lambda} H$ associated
to this matched pair is called a \textit{generalized quantum
double} and it will be denoted by $D_{\lambda} (A, H)$.
\end{examples}

A Hopf algebra $E$ \emph{factorizes} through two Hopf algebras $A$
and $H$ if there exists injective Hopf algebra maps $i : A \to E $
and $j : H\to E$ such that the map
$$
A \ot H \to E, \quad a \ot h \mapsto i(a) j(h)
$$
is bijective. The main theorem which characterizes the bicrossed
product is the next theorem due to Majid \cite[Theorem
7.2.3]{majid2}. The normal version of it was recently proven in
\cite[Proposition 2.2]{burCentral} and more generally in
\cite[Theorem 2.1]{agorecia}. First we recall that a Hopf
subalgebra $A$ of a Hopf algebra $E$ is called \textit{normal} if
$x_{(1)}aS(x_{(2)}) \in A$ and $S(x_{1})ax_{(2)} \in A$, for all
$x \in E$ and $a \in A$.

\begin{theorem}\thlabel{carMaj}
Let $A$, $H$ be two Hopf algebras. A Hopf algebra $E$ factorizes
through $A$ and $H$ if and only if there exists a matched pair of
Hopf algebras $(A, H, \triangleleft, \triangleright)$ such that $E
\cong A \bowtie H$, an isomorphism of Hopf algebras.

Furthermore, a Hopf algebra $E$ factorizes through a normal Hopf
subalgebra $A$ and a Hopf subalgebra $H$ if and only if $E$ is
isomorphic as a Hopf algebra to a semidirect product $A \# \, H$.
\end{theorem}

\begin{proof} The bicrossed product $A \bowtie H$ factorizes
through $A$ and $H$ via the canonical maps $i_A : A \to A \bowtie
H$, $i_A (a) = a \bowtie 1_H$ and $i_H : H \to A \bowtie H$, $i_H
(h) = 1_ A \bowtie h$. Conversely, assume that $E$ factorizes
through $A$ and $H$ and we view $A$ and $H$ as Hopf subalgebras of
$E$ via the identification $A \cong i(A)$ and $H \cong j(H)$. For
any $a\in A$ and $h\in H$ we view $ha \in E$; as $E$ factorizes
through $A$ and $H$, we can find an unique element $\sum_{i=1}^t
a_i^{(a, h)} \ot h_i^{(a, h)} \in A \ot H$ such that
$$
ha = \sum_{i=1}^t \, a_i^{(a, h)} \,\, h_i^{(a, h)}
$$
We define the maps $\triangleleft : H \otimes A \rightarrow H$ and
$\triangleright: H \otimes A \rightarrow A$ via:
\begin{equation}\eqlabel{deduc}
h \triangleleft a := \sum_{i=1}^t \, \varepsilon_A (a_i^{(a, h)})
\, h_i^{(a, h)},  \, \quad \, h \triangleright a : = \sum_{i=1}^t
\, \varepsilon_H(h_i^{(a, h)}) \, a_i^{(a, h)}
\end{equation}
Then $(A, H, \triangleleft, \triangleright)$ is a matched pair of
Hopf algebras and the multiplication map $ A\bowtie H \to E$,
$a\bowtie h \mapsto ah$ is an isomorphism of Hopf algebras. The
details are proven in \cite[Theorem 7.2.3]{majid2}. The final
statement is \cite[Proposition 2.2]{burCentral}.
\end{proof}

\thref{carMaj} proves that the factorization problem for Hopf
algebras can be restated in a computational manner: \textit{Let
$A$ and $H$ be two given Hopf algebras. Describe the set of all
matched pairs $(A, H, \triangleright, \triangleleft)$ and classify
up to an isomorphism all bicrossed products $A \bowtie \, H$.}

The Hopf algebras $A$ and $H$ play a symmetric role in a bicrossed
product $A \bowtie H$. In other words, finding all matched pairs
$(H, A, \triangleright', \triangleleft')$ reduces in fact to
finding all matched pairs $(A, H, \triangleright, \triangleleft)$.
In the case where both Hopf algebras $A$ and $H$ have bijective
antipode we indicate an explicit way of constructing a matched
pair $(H, A, \triangleright', \triangleleft')$ out of a given
matched pair $(A, H, \triangleright, \triangleleft)$ such that the
associated Hopf algebras $A \bowtie H$ and $H\bowtie ' A$ are
isomorphic. It is the Hopf algebra version of \cite[Proposition
2.5]{acim} proved for matched pairs of groups.

\begin{proposition}\prlabel{trecere}
Let $(A, H, \triangleright, \triangleleft)$ be a matched pair of
Hopf algebras with bijective antipodes. We define the actions
$\triangleright': A \ot H \to H$ and $\triangleleft' : A\ot H \to
A$ by
\begin{eqnarray*}
a \triangleright' h &:=& S_H \Bigl( S_H^{-1} (h_{(1)})
\triangleleft S_A^{-1} (a_{(1)}) \Bigl) \, \triangleleft \,
S_A\Bigl( S_H^{-1} (h_{(2)}) \triangleright
S_A^{-1} (a_{(2)}) \Bigl)\\
a \triangleleft' h &:=& S_H \Bigl( S_H^{-1} (h_{(1)})
\triangleleft S_A^{-1} (a_{(1)}) \Bigl) \, \triangleright \,
S_A\Bigl( S_H^{-1} (h_{(2)}) \triangleright S_A^{-1} (a_{(2)})
\Bigl)
\end{eqnarray*}
for all $a\in A$ and $h\in H$. Then $(H, A, \triangleright',
\triangleleft')$ is a matched pair of Hopf algebras and there
exists an isomorphism of Hopf algebras $A \bowtie H \cong H
\bowtie' A$, where $H \bowtie' A$ is the bicrossed product
associated to $(H, A, \triangleright', \triangleleft')$.
\end{proposition}

\begin{proof} The proof is a straightforward verification
based on \thref{carMaj} as follows: let $E:= A\bowtie H$ be the
bicrossed product associated to the matched pair $(A, H,
\triangleright, \triangleleft)$. Then $E$ factorizes through $A$
and $H$; hence $E$ also factorizes through $H$ and $A$. Next we
write down the cross relation in the bicrossed product $A\bowtie
H$
$$
(1_A \bowtie h) \cdot (a \bowtie 1_H) = h_{(1)} \triangleright
a_{(1)} \bowtie h_{(2)} \triangleleft a_{(2)}
$$
for $a = S_A^{-1} (a')$ and $ h = S_H^{-1} (h')$, with $a'\in A$
and $h'\in H$ and then we apply $S_{A\bowtie H}$ to it. We shall
obtain an expression for $ a' \bowtie h'$ in terms of the actions
$\triangleright$ and $\triangleleft$ as follows:
\begin{eqnarray*}
&& a' \bowtie h' = S_H \Bigl( S_H^{-1} (h'_{(1)}) \triangleleft
S_A^{-1} (a'_{(1)}) \Bigl) \, \triangleright \,  S_A \Bigl(
S_H^{-1} (h'_{(3)}) \triangleright
S_A^{-1} (a'_{(3)}) \Bigl) \, \bowtie \\
&\bowtie& S_H \Bigl( S_H^{-1} (h'_{(2)}) \triangleleft  S_A^{-1}
(a'_{(2)}) \Bigl) \, \triangleleft \, S_A \Bigl( S_H^{-1}
(h'_{(4)}) \triangleright S_A^{-1} (a'_{(4)}) \Bigl)
\end{eqnarray*}
Now, the actions $\triangleright'$ si $\triangleleft'$ of the new
matched pair $(H, A, \triangleright', \triangleleft')$ are
obtained using \equref{deduc}, i.e. we first apply $\varepsilon_A
\ot {\rm Id}$ and then ${\rm Id} \ot \varepsilon_H$ to this
formula. Finally, the k-linear map $\varphi : A\ot H \to H\ot A$,
$\varphi (a \ot h): = S_H (h) \ot S_A (a)$ is bijective and we can
easily prove that $\psi : A \bowtie H \to H \bowtie' A$, $\psi :=
\varphi \circ S_{A\bowtie H} $ is an isomorphism of Hopf algebras.
\end{proof}

\section{The morphisms between two bicrossed products}\selabel{sectth}

In order to describe the morphisms between two arbitrary bicrossed
products $A \bowtie H$ and $A' \bowtie ' H'$ we need the
following:

\begin{lemma}\lelabel{coalgmap}
Let $C$, $D$ and $E$ be three coalgebras. Then there exists a
bijective correspondence between the set of all morphisms of
coalgebras $\alpha : C \to D \ot E$ and the set of all pairs $(u,
p)$, where $u : C \to D$ and $p : C \to E$ are morphisms of
coalgebras satisfying the symmetry condition
\begin{equation}\eqlabel{symcoal}
p (c_{(1)}) \ot  u (c_{(2)}) = p (c_{(2)}) \ot u (c_{(1)})
\end{equation}
for all $c \in C$. Under the above correspondence the morphism of
coalgebras $\alpha : C \to D \ot E$ corresponding to the pair $(u,
p)$ is given by:
\begin{equation}\eqlabel{coalgmaps}
\alpha (c) =  u (c_{(1)}) \ot p (c_{(2)})
\end{equation}
for all $c \in C$.
\end{lemma}

\begin{proof}
Let $\alpha : C \to D \ot E$ be a morphism of coalgebras. We adopt
the temporary notation $\alpha (c) = c_{ \{0\}} \ot c_{ \{1\}} \in
D \ot E$ (summation understood). Define the following two linear
maps:
\begin{eqnarray*}
u: C \rightarrow D, \quad u := ({\rm Id}\ot \varepsilon_E) \circ
\alpha, \,\,\,\,\, {\rm i.e. } \,\, \quad u (c) = \varepsilon_E (c_{ \{1\}}) c_{ \{0\}}\\
p: C \rightarrow E, \quad p := (\varepsilon_D \ot {\rm Id}) \circ
\alpha, \,\,\,\,\, {\rm i.e. } \,\, \quad p (c) = \varepsilon_D
(c_{ \{0\}}) c_{ \{1\}}
\end{eqnarray*}
for all $c\in C$. Then $u$ and $p$ are coalgebra maps as
compositions of coalgebra maps. Since $\alpha$ is a coalgebra map
we have $\Delta_{D\ot E} \circ \alpha(c) = (\alpha \ot \alpha)
\circ \Delta_C(c)$ for any $c\in C$, which is equivalent to:
\begin{equation}\eqlabel{coalgmaps22}
c_{\{0\}(1)} \ot c_{\{1\}(1)} \ot c_{\{0\}(2)} \ot c_{\{1\}(2)} =
c_{ (1)\{0\}} \ot c_{ (1)\{1\}} \ot c_{ (2)\{0\}} \ot c_{
(2)\{1\}}
\end{equation}
If we apply ${\rm Id} \ot \varepsilon_E \ot \varepsilon_D \ot {\rm
Id}$ to \equref{coalgmaps22} we obtain $c_{ \{0\}} \ot c_{ \{1\}}
= u (c_{(1)}) \ot p (c_{(2)})$ that is \equref{coalgmaps} holds.
With this form for $\alpha$ and having in mind that $u$ and $p$
are coalgebra maps the equation \equref{coalgmaps22} takes the
form:
$$
u (c_{(1)}) \ot p (c_{(3)}) \ot u (c_{(2)}) \ot p (c_{(4)}) = u
(c_{(1)}) \ot p (c_{(2)}) \ot u (c_{(3)}) \ot p (c_{(4)})
$$
Applying $\varepsilon_D \ot {\rm Id} \ot {\rm Id} \ot
\varepsilon_E $ to the above identity we obtain the symmetry
condition \equref{symcoal}. Conversely, it is easy to see that any
map $\alpha$ given by \equref{coalgmaps}, for some coalgebra maps
$u$ and $p$ satisfying \equref{symcoal}, is a coalgebra map and it
follows from the proof that the correspondence is bijective.
\end{proof}

Now we can describe all Hopf algebra morphisms between two bicrossed
products.

\begin{theorem}\thlabel{toatemorf}
Let $(A, H, \triangleright, \triangleleft)$ and $(A', H',
\triangleright', \triangleleft')$ be two matched pairs of Hopf
algebras. Then there exists a bijective correspondence between the
set of all morphisms of Hopf algebras $\psi : A \bowtie H \to A'
\bowtie ' H' $ and the set of all quadruples $(u, p, r, v)$, where
$u: A \to A'$, $p: A \to H'$, $r: H \rightarrow A'$, $v: H
\rightarrow H'$ are unitary coalgebra maps satisfying the
following compatibility conditions:
\begin{eqnarray}
u(a_{(1)}) \ot p(a_{(2)}) &{=}& u(a_{(2)}) \ot p(a_{(1)})\eqlabel{C1}\\
r(h_{(1)}) \ot v(h_{(2)}) &{=}& r(h_{(2)}) \ot v(h_{(1)})\eqlabel{C2}\\
u(ab) &{=}& u(a_{(1)}) \, \bigl( p (a_{(2)}) \triangleright' u(b) \bigl)\eqlabel{C3}\\
p(ab) &{=}& \bigl( p (a) \triangleleft' u(b_{(1)}) \bigl) \, p (b_{(2)})\eqlabel{C4}\\
r(hg) &{=}& r(h_{(1)}) \, \bigl(v(h_{(2)}) \triangleright'
r(g)\bigl)\eqlabel{C5}\\
v(hg) &{=}& \bigl(v(h) \triangleleft' r(g_{(1)})\bigl) \,
v(g_{(2)})\eqlabel{C6}\\
r(h_{(1)}) \bigl(v(h_{(2)}) \triangleright' u(b) \bigl) &{=}& u
(h_{(1)} \triangleright b_{(1)}) \, \Bigl( p (h_{(2)}
\triangleright b_{(2)}) \triangleright' r(h_{(3)} \triangleleft
b_{(3)})
\Bigl) \eqlabel{C7}\\
\bigl (v(h) \triangleleft' u(b_{(1)}) \bigl) \, p (b_{(2)}) &{=}&
\Bigl( p (h_{(1)} \triangleright b_{(1)}) \triangleleft '
r(h_{(2)} \triangleleft b_{(2)}) \Bigl) \, v (h_{(3)}
\triangleleft b_{(3)}) \eqlabel{C8}
\end{eqnarray}
for all $a$, $b \in A$, $g$, $h \in H$.

Under the above correspondence the morphism of Hopf algebras
$\psi: A \bowtie H \to A' \bowtie ' H' $ corresponding to $(u, p,
r, v)$ is given by:
\begin{equation}\eqlabel{morfbicros}
\psi(a \bowtie h) = u(a_{(1)}) \, \bigl( p(a_{(2)})
\triangleright' r(h_{(1)}) \bigl) \,\, \bowtie' \, \bigl(
p(a_{(3)}) \triangleleft' r(h_{(2)}) \bigl) \, v(h_{(3)})
\end{equation}
for all $a \in A$ and $h\in H$.
\end{theorem}

\begin{proof}
Let $\psi: A \bowtie H \to A' \bowtie ' H' $ be a morphism of Hopf
algebras. We define
$$
\alpha : A \to A' \bowtie ' H', \quad \alpha (a) := \psi (a
\bowtie 1_H)
$$
$$
\beta : H \to A' \bowtie ' H', \quad \beta (h) := \psi (1_A
\bowtie h)
$$
Then $\alpha : A \to A' \ot  H'$ and $\beta : H \to A' \ot H'$ are
unitary morphisms of coalgebras as compositions of such maps (we
recall that the coalgebra structure on $A' \bowtie ' H'$ is the
tensor product of coalgebras $A' \ot  H'$) and
\begin{equation}\eqlabel{ecu3}
\psi (a \bowtie h) = \psi ( (a \bowtie 1_H) (1_A \bowtie h) ) =
\psi (a \bowtie 1_H) \psi (1_A \bowtie h) = \alpha (a) \, \beta(h)
\end{equation}
for all $a\in A$ and $h\in H$. It follows from \leref{coalgmap}
applied to $\alpha$ and $\beta$ that there exist four coalgebra
maps $u: A \to A'$, $p: A \to H'$, $r: H \rightarrow A'$, $v: H
\rightarrow H'$ such that
\begin{equation}\eqlabel{ecu33}
\alpha (a) =  u (a_{(1)}) \ot p (a_{(2)}), \qquad \beta (h) = r
(h_{(1)}) \ot v (h_{(2)})
\end{equation}
and the pairs $(u, p)$ and $(r, v)$ satisfy the symmetry
conditions \equref{C1} and \equref{C2}. Explicitly $u$, $p$, $r$
and $v$ are defined by
$$
u(a) = (({\rm Id}\ot \varepsilon_{H'}) \circ \psi ) (a \bowtie
1_H), \quad p(a) = ((\varepsilon_{A'} \ot {\rm Id}) \circ \psi )
(a \bowtie 1_H)
$$
$$
r(h) =  (({\rm Id}\ot \varepsilon_{H'}) \circ \psi ) (1_A \bowtie
h), \quad v(h) = ((\varepsilon_{A'} \ot {\rm Id}) \circ \psi )
(1_A \bowtie h)
$$
for all $a\in A$ and $h\in H$. All these maps are unitary
coalgebra maps. Now, for any $a\in A$ and $h\in H$ we have:
\begin{eqnarray*}
\psi (a \bowtie h) &=& \alpha (a) \beta(h)\\
&=& \bigl( u (a_{(1)}) \bowtie' p (a_{(2)}) \bigl) \cdot \bigl( r
(h_{(1)}) \bowtie' v (h_{(2)}) \bigl)\\
&=& u(a_{(1)}) \, \bigl( p(a_{(2)}) \triangleright' r(h_{(1)})
\bigl) \,\, \bowtie' \, \bigl( p(a_{(3)}) \triangleleft'
r(h_{(2)}) \bigl) \, v(h_{(3)})
\end{eqnarray*}
i.e. \equref{morfbicros} also holds. Thus any bialgebra map $\psi:
A \bowtie H \to A' \bowtie ' H' $ is uniquely determined by the
formula \equref{ecu3} for some unitary coalgebra maps $\alpha: A
\to A' \ot  H'$ and $\beta : H \to A' \ot  H'$ or, equivalently,
in the more explicit form given by \equref{morfbicros}, for some
unique quadruple of unitary coalgebra maps $(u, p, r, v)$.

Now, a map $\psi$ given by \equref{ecu3} is a morphism of algebras
if and only if $\alpha : A \to A' \bowtie ' H'$ and $\beta : H \to
A' \bowtie ' H'$ are algebra maps and the following commutativity
relation holds
\begin{equation}\eqlabel{ecu11}
\beta (h) \, \alpha (b) = \alpha (h_{(1)} \triangleright b_{(1)} )
\, \beta (h_{(2)} \triangleleft b_{(2)} )
\end{equation}
for all $h\in H$ and $b\in A$. Indeed, if $\psi$ is an algebra map
then $\alpha$ and $\beta$ are algebra maps as compositions of
algebra maps. On the other hand:
$$
\psi (a \bowtie h) \psi (b \bowtie g) = \alpha (a) \beta (h)
\alpha (b) \beta (g)
$$
and
$$
\psi \bigl( (a \bowtie h) (b \bowtie g) \bigl ) = \alpha (a)
\alpha (h_{(1)} \triangleright b_{(1)} ) \beta (h_{(2)}
\triangleleft b_{(2)} ) \beta (g)
$$
Hence, the relation \equref{ecu11} follows by taking $a = 1_A$ and
$g = 1_H$ in the identity above. Conversely is obvious. Now, we
write down the explicit conditions for $\alpha$ and $\beta$ to be
algebra maps. First, we prove that $\alpha : A \to A' \bowtie '
H'$, $\alpha (a) = u(a_{(1)}) \bowtie' p(a_{(2)})$ is an algebra
map if and only if \equref{C3} and \equref{C4} hold. Indeed,
$\alpha (ab) = \alpha (a) \alpha (b)$ is equivalent to:
\begin{equation}\eqlabel{a10}
u(a_{(1)} b_{(1)} ) \bowtie' p(a_{(2)} b_{(2)} ) = u(a_{(1)})
\bigl( p(a_{(2)}) \triangleright' u(b_{(1)}) \bigl) \, \bowtie' \,
\bigl( p(a_{(3)}) \triangleleft' u(b_{(2)}) \bigl) p(b_{(3)})
\end{equation}
If we apply ${\rm Id} \ot \varepsilon_{H'}$ to \equref{a10} we
obtain \equref{C3} while if we apply $ \varepsilon_{A'} \ot {\rm
Id}$ to \equref{a10} we obtain \equref{C4}. Conversely is obvious.
In a similar way we can show that $\beta : H \to A' \bowtie' H'$,
$\beta (h) = r (h_{(1)}) \bowtie' v (h_{(2)})$ is an algebra map
if and only if  \equref{C5} and \equref{C6} hold. Finally, we
prove that the commutativity relation \equref{ecu11} is equivalent
to \equref{C7} and \equref{C8}. Indeed, using the expressions of
 $\alpha$ and $\beta$ in terms of $(u, p)$
and respectively $(r, v)$, the equation \equref{ecu11} is
equivalent to:
\begin{eqnarray*}
&&r(h_{(1)}) \bigl(v(h_{(2)}) \triangleright' u(b_{(1)}) \bigl)
\,\,
\bowtie' \,\, \bigl (v(h_{(3)}) \triangleleft' u(b_{(2)}) \bigl) \, p (b_{(3)}) = \\
&& u (h_{(1)} \triangleright b_{(1)}) \, \Bigl( p (h_{(2)}
\triangleright b_{(2)}) \triangleright' r(h_{(3)} \triangleleft
b_{(3)})
\Bigl) \,\,\, \bowtie  \\
&& \Bigl( p (h_{(4)} \triangleright b_{(4)}) \triangleleft '
r(h_{(5)} \triangleleft b_{(5)}) \Bigl) \, v (h_{(6)}
\triangleleft b_{(6)})
\end{eqnarray*}
If  we apply ${\rm Id} \ot \varepsilon_{H'}$ to the above identity
we obtain \equref{C7} while if we apply $ \varepsilon_{A'} \ot
{\rm Id}$ to it we get \equref{C8}. Conversely, the commutativity
condition \equref{ecu11} follows straightforward from \equref{C7}
and \equref{C8}.

To conclude, we have proved that a bialgebra map $\psi: A \bowtie
H \to A' \bowtie ' H' $ is uniquely determined by a quadruple $(u,
p, r, v)$ of unitary coalgebra maps satisfying the compatibility
conditions \equref{C1}-\equref{C8} such that $\psi: A \bowtie H
\to A' \bowtie ' H' $ is given by \equref{morfbicros}.

Conversely, the fact that $\psi$ given by \equref{morfbicros} is a
morphism of bialgebras, for some $(u, p, r, v)$ satisfying
\equref{C1}-\equref{C8} is straightforward and follows directly
from the proof. The fact that $\psi$ is an algebra map is proven
above. $\psi$ is also a coalgebra map as a composition of
coalgebra maps. More precisely, the equation \equref{morfbicros}
can be written in the equivalent form \equref{ecu3} namely, $\psi
= m_{A' \bowtie ' H'} \circ (\alpha \ot \beta)$, where $m_{A'
\bowtie ' H'}$ is the multiplication map on the bicrossed product
$A' \bowtie ' H'$ and $\alpha$, $\beta$ are the unitary coalgebra
maps obtained from $(u, p, r, v)$ via the formulas \equref{ecu33}.
\end{proof}

In the next corollary $A\#H$ will be a semi-direct product of Hopf
algebras constructed in (1) of \exref{exempleban} as a special
case of a bicrossed product.

\begin{corollary}\colabel{endosmh}
Let $A\# H$ and $A' \#' H'$ be two semi-direct products of Hopf
algebras associated to two left actions $\triangleright : H\ot A
\to $ and $\triangleright' : H'\ot A' \to A'$. Then there exists a
bijection between the set of all morphisms of Hopf algebras $\psi
: A \# H \to A' \#' H'$ and the set of all quadruples $(u, p, r,
v)$, where $u: A \to A'$, $r: H \rightarrow A'$ are unitary
coalgebra maps, $p: A \to H'$, $v: H \rightarrow H'$ are morphism
of Hopf algebras satisfying the following compatibility
conditions:
\begin{eqnarray}
u(a_{(1)}) \ot p(a_{(2)}) &{=}& u(a_{(2)}) \ot p(a_{(1)})\eqlabel{C1ab}\\
r(h_{(1)}) \ot v(h_{(2)}) &{=}& r(h_{(2)}) \ot v(h_{(1)})\eqlabel{C2ab}\\
u(ab) &{=}& u(a_{(1)}) \, \bigl( p (a_{(2)}) \triangleright ' u(b) \bigl)\eqlabel{C3ab}\\
r(hg) &{=}& r(h_{(1)}) \, \bigl(v(h_{(2)}) \triangleright ' r(g)\bigl)\eqlabel{C5ab}\\
r(h_{(1)}) \bigl(v(h_{(2)}) \triangleright' u(b) \bigl) &{=}& u
(h_{(1)} \triangleright b_{(1)}) \, \Bigl( p (h_{(2)}
\triangleright b_{(2)}) \triangleright ' r(h_{(3)})
\Bigl) \eqlabel{C7ab}\\
v(h) \, p(b) &{=}& p (h_{(1)} \triangleright b ) \, v (h_{(2)})
\eqlabel{C8ab}
\end{eqnarray}
for all $a$, $b \in A$, $g$, $h \in H$.

Under the above bijection the morphism of Hopf algebras $\psi: A
\# H \to A' \#' H'$ corresponding to $(u, p, r, v)$ is given by:
\begin{equation}\eqlabel{morfbicrossmab}
\psi(a \# h) = u(a_{(1)}) \, \bigl( p(a_{(2)}) \triangleright'
r(h_{(1)}) \bigl) \,\, \#' \,  p(a_{(3)}) \, v(h_{(2)})
\end{equation}
for all $a \in A$ and $h\in H$.
\end{corollary}

\begin{proof} We apply \thref{toatemorf} in the case that
the two right actions $\triangleleft'$,  $\triangleleft$ are both
the trivial actions. In this case, \equref{C4} is equivalent to
$p: A \to H'$ being an algebra map, i.e. $p$ is a Hopf algebra map
while \equref{C6} is equivalent to $v: H \to H'$ being an algebra
map, i.e. $v$ is a Hopf algebra map. Finally,
\equref{C1}-\equref{C3}, \equref{C5}, \equref{C7} and \equref{C8}
take the simplified forms \equref{C1ab}-\equref{C3ab},
\equref{C5ab}, \equref{C7ab} and \equref{C8ab} respectively.
\end{proof}

Using \thref{toatemorf} we can completely describe $\End (A
\bowtie H)$, the space of all Hopf algebra endomorphisms of an
arbitrary bicrossed product. In particular, we obtain a
description of $\End (D (H))$, the group of all Hopf algebra
endomorphisms of a Drinfel'd double $D(H)$. If $H = k[G]$, for a
finite group $G$, the group $\End (D(k[G]))$ of all Hopf algebra
endomorphisms is described bellow in full details.

Let $G$ be a finite group. The Drinfel'd double of $k[G]$ is
described as follows. Take $\{ e_g \}_{g\in G}$ to be a dual basis
for the basis $\{g \}_{g\in G}$ of $k[G]$. $A := ( k[G]^* )^{\rm
cop}$ is a Hopf algebra with the multiplication and the
comultiplication given by
$$
e_g \cdot e_h := \delta_{g, h} e_g, \quad \Delta_A (e_g) :=
\sum_{x\in G} \, e_x \ot e_{gx^{-1}}, \quad 1_A := \sum_{g\in G}
\, e_g, \quad \varepsilon_A (e_g) := \delta_{g, 1_G}
$$
for all $g$, $h\in G$, where $\delta_{( - , \,-)}$ is the
Kronecker delta. The left action of $H = k[G]$ on $A= ( k[G]^*
)^{\rm cop}$ given by \equref{dubluact} takes the form
\begin{equation}\eqlabel{actdudr}
g \triangleright e_h := e_{g h g^{-1}} \quad {\rm or} \quad (g
\triangleright f) (z) : = f (g^{-1} z g)
\end{equation}
for all $g$, $h$, $z\in G$ and $f \in (k[G]^*)^{\rm cop}$. Using
the action \equref{actdudr}, we have that $D(k[G]) = (k[G]^*)^{\rm
cop} \# k[G]$. Thus, $D(k[G])$ is the Hopf algebra having the
basis $\{ e_h \# g \}_{g, h\in G}$ with the multiplication and the
comultiplication given by
$$
(e_h \, \# \, g) \, \cdot \, (e_x \, \# \,  y) = \delta_{h, g x
g^{-1}} \,\, e_h \, \# \, g y, \qquad  \Delta (e_h \, \# \, g) =
\sum_{x\in G} \, \, (e_x \, \# \, g) \ot (e_{h x^{-1}} \, \# \, g)
$$
for all $h$, $g$, $x$ and $y\in G$.

The following gives the parametrization of all Hopf algebra
morphisms between two Drinfel'd doubles associated to two finite
groups $G$ and $H$. In particular, if $G = H$, the space $\End
\bigl(D(k[G])\bigl)$ of all Hopf algebra endomorphisms is
described.

\begin{corollary}\colabel{endoDH}
Let $G$, $H$ be two finite groups. Then any Hopf algebra morphism
$\psi: D(k[G]) \to D(k[H]) $ between the Drinfel'd doubles has the
form\footnote{We denote by $\theta(? , \, y) \delta_{(?, \,
aba^{-1})} \in k[H]^*$ the $k$-linear map sending any $z \in H$ to
$\theta(z , \, y) \delta_{\{z, \, aba^{-1}\}}$.}
\begin{equation} \eqlabel{morfisdkg}
\psi ( e_g \# g') = \sum_{x\in G, \, y, z \in H} \,\, \lambda (g
x^{-1}, \, z) \, \omega (g', \, y) \, \theta(? , \, x) \delta_{(?,
\, z y z^{-1})} \,\# \, z \, v(g')
\end{equation}
for all $g$, $g' \in G$, where $(\lambda, \omega, \theta, v)$ is a
quadruple such that $v: G \to H$ is a morphism of groups, $\theta
: H\times G \to k$, $\lambda : G\times H \to k$, $\omega : G\times
H \to k$ are three maps satisfying the following compatibilities:
\begin{eqnarray}
\theta (1, g) &=& \delta_{g, 1_G} \eqlabel{dr1a}\\
\sum_{x\in G}\, \theta (h, x) &=& 1 \eqlabel{dr1}\\
\theta (h h', g) &=& \sum_{x\in G} \, \theta (h, x) \, \theta (h',
x^{-1}g) \eqlabel{dr2}\\
\omega (1, g) &=& 1 \eqlabel{dr3a}\\
\omega(g, h h') &=& \omega (g, h) \, \omega(g, h') \eqlabel{dr3}\\
\sum_{y\in H} \, \lambda (g, y) &=& \delta_{g, 1} \eqlabel{dr4}\\
\sum_{x\in G} \, \lambda (x, h) &=& \delta_{1, h} \eqlabel{dr4aa}
\end{eqnarray}
\begin{eqnarray}
\sum_{y\in H} \, \lambda (g, y) \, \lambda (g' , y^{-1} h) &=&
\delta_{g, g'} \, \lambda (g, h) \eqlabel{dr5}\\
\sum_{x\in G} \, \lambda (x, h) \, \lambda (g x^{-1}, h' ) &=&
\delta_{h, h'} \, \lambda (g, h) \eqlabel{dr6}\\
\sum_{x\in G} \,  \theta (h, x) \, \lambda (g x^{-1}, h' ) &=&
\sum_{x\in G} \, \theta (h, x) \, \lambda (x^{-1} g , h' ) \eqlabel{dr7}\\
\sum_{x \in G, \, y \in H} \, \lambda (g x^{-1}, y) \, \theta (h,
x) \, \theta (y^{-1}h y, g' ) &=& \delta_{g, g'} \, \theta (h, g)
\eqlabel{dr8}\\
\omega (g, h) \, \omega \bigl(g' , v(g)^{-1}h v(g) \bigr) &=&
\omega
(g g' , h)  \eqlabel{dr9}\\
\sum_{x\in G, \, y \in H} \, \theta (h, x g^{-1}) \, \omega (g,
y^{-1} h y) \lambda (g g' x^{-1}, y) &=& \omega (g', h) \, \theta
\bigl(v(g)^{-1} h v(g), g' \bigl) \eqlabel{dr10}\\
\lambda \bigl( g g' g^{-1}, v(g) \, h \, v(g)^{-1} \bigl) &=&
\lambda (g', h) \eqlabel{dr11}
\end{eqnarray}
for all $g$, $g'\in G$, $h$, $h'\in H$.

The correspondence $\psi \leftrightarrow (\lambda, \omega, \theta,
v)$ between the set of all Hopf algebra morphisms $\psi: D(k[G])
\to D(k[H])$ and the set of all maps $(\lambda, \omega, \theta,
v)$ satisfying the compatibility conditions \equref{dr1a} -
\equref{dr11} is bijective.
\end{corollary}

\begin{proof}
It follows by a direct computation from \coref{endosmh} applied
for $A = ( k[G]^* )^{\rm cop}$, $H = k[H]$ and the left action
$\triangleright$ given by \equref{actdudr}. We shall indicate only
a sketch of the proof, the details being left to the reader. First
we should notice that any Hopf algebra map $v : k[G] \to k[H]$ is
uniquely determined by a morphism of the groups which will be
denoted also by $v: G \to H$. We shall prove now that any unitary
coalgebra map $ u: (k[G]^*)^{\rm cop} \to (k[H]^*)^{\rm cop}$ is
given by the formula
$$
u(e_g) ( h ) = \theta (h, g)
$$
for all $g\in G$, $h\in H$ and for a unique map $\theta: H\times G
\to k$ satisfying \equref{dr1a}-\equref{dr2}. Indeed, first we
notice that a unitary coalgebra map $ u: (k[G]^*)^{\rm cop} \to
(k[H]^*)^{\rm cop}$ is in fact the same as a unitary coalgebra map
$ u: k[G]^* \to k[H]^*$. From the duality algebras/coalgebras any
such map $u$ is uniquely implemented by a map of augmented
algebras $f: k[H] \to k[G]$ (i.e. $f$ is an algebra map and
$\varepsilon_{k[G]} \circ f = \varepsilon_{k[H]}$) by the formula
$u = f^*$, i.e. $u(e_h) = e_h \circ f$, for all $h\in G$. Now, any
$k$-linear map $f: k[H] \to k[G]$ is uniquely defined by a map
$\theta: H\times G \to k$ as follows:
$$
f(h) = \sum_{x\in G} \, \theta (h, x) x
$$
for any $g\in H$. We can easily prove that such a linear map is an
endomorphism of $k[G]$ as augmented algebras if and only if the
compatibility conditions \equref{dr1a}-\equref{dr2} hold. Now, any
$k$-linear map $r : k[G] \to (k[H]^*)^{\rm cop}$ is implemented by
a unique map $\omega : G\times H \to k$ such that
$$
r(g) = \sum_{y\in H} \, \omega (g, y) \, e_y
$$
for all $g\in G$. Such a map $r = r_{\omega}$ is an unitary
morphism of coalgebras if and only if $\omega$ satisfies the
compatibility conditions \equref{dr3a}-\equref{dr3}. Finally, any
$k$-linear map $p: (k[G]^*)^{\rm cop} \to k[H]$ is given in a
unique way by a map $\lambda : G\times H \to k$ such that
$$
p(e_g) = \sum_{y\in H} \, \lambda (g, y) y
$$
for all $g\in G$. Such a map $p = p_{\lambda}$ is a morphism of
Hopf algebras if and only if $\lambda$ satisfies the compatibility
conditions \equref{dr4}-\equref{dr6}.

Hence we have described the set of data $(u, p, r, v)$ of
\coref{endosmh}. As $H = k[H]$ is cocommutative the compatibility
condition \equref{C2ab} is trivially fulfilled. Moreover, by a
straightforward computation the compatibility conditions
\equref{C1ab} and \equref{C3ab}-\equref{C8ab} take the form
\equref{dr7}-\equref{dr11}. Now, by a direct computation we can
show that the expression of the morphism $\psi: D(k[G]) \to
D(k[H])$ given by \equref{morfbicrossmab} takes the following
simplified form:
\begin{eqnarray*}
\psi ( e_g \# g') &{=}& \sum_{a, b, c \in H, \, x, y \in G} \,\,
\lambda (x y^{-1}, \, a) \, \lambda (g x^{-1}, \, c) \, \omega
(g', \, b) \,
\theta(? , \, y) \delta_{(?, \, aba^{-1})} \,\# \, c \, v(g')\\
&{=}& \sum_{a, b, c\in H, y\in G} \, \bigl(\sum_{x \in G} \lambda
(x y^{-1}, \, a) \, \lambda (g x^{-1}, \, c) \bigl) \, \omega (g',
\,
b) \, \theta(? , \, y) \delta_{(?, \, aba^{-1})} \,\# \, c \, v(g')\\
&\stackrel{\equref{dr6}} {=}& \sum_{a, b, c\in H, y \in G} \,
\delta_{a, c} \lambda (g y^{-1}, c) \, \omega (g', \,
b) \, \theta(? , \, y) \delta_{(?, \, aba^{-1})} \,\# \, c \, v(g')\\
&{=}& \sum_{a, b \in H, y \in G} \,\, \lambda (g y^{-1}, \, a) \,
\omega (g', \, b) \, \theta(? , \, y) \delta_{(?, \, aba^{-1})}
\,\# \, a \, v(g')
\end{eqnarray*}
i.e., by interchanging the summation indices, we proved that
\equref{morfisdkg} holds.
\end{proof}

\section{The classification of bicrossed products} \selabel{casbic}
\thref{toatemorf} can be used to indicate when two arbitrary
bicrossed products $A \bowtie H$ and $A'\bowtie' H'$ are
isomorphic. Hence it gives the answer to the classification part
of the factorization problem for Hopf algebras and will be used in
its full generality for explicit examples in \seref{secexemple}.
In general, since the result is very technical and not so
transparent, we restrict ourselves to a special kind of
classification, namely the one that stabilizes one of the terms of
the bicrossed product. As explained in the introduction the
classification of bicrossed products up to an isomorphism that
stabilizes one of the terms has two motivations: the first one is
the cohomological point of view which descends to the
classification theory of group extensions and the second one is
the problem of describing and classifying the $A$-forms of a Hopf
algebra extensions from descent theory \cite{am12}.

Let $A$ and $H$ be two Hopf algebras. We shall classify up to an
isomorphism that stabilizes $A$ the set of all Hopf algebras $E$
that factorize through $A$ and $H$. It follows from \thref{carMaj}
that we have to classify all bicrossed products $A \bowtie H$
associated to all possible matched pairs of Hopf algebras $(A, H,
\triangleright, \triangleleft)$. First we need to recall from
\cite{AD} the following concept that will play a crucial role in
the paper.

\begin{definition}\delabel{lazycocdef}
Let $A$ and $H$ be two Hopf algebras. A coalgebra map $r: H
\rightarrow A$ is called \emph{cocentral} if the following
compatibility holds:
\begin{equation}\eqlabel{0aa}
r(h_{(1)}) \ot h_{(2)} = r(h_{(2)}) \ot h_{(1)}
\end{equation}
for all $h\in H$. We denote by $CoZ^{1}(H, A)$ the group with
respect to the convolution product of all unitary cocentral maps
from $H$ to $A$.
\end{definition}

\begin{remark}\relabel{involdisp}
If $r \in CoZ^{1}(H, A)$ then $S_A^2 \circ r = r$. Indeed, since
$r$ is a coalgebra map, the inverse of $r$ in the group
$CoZ^{1}(H, A)$ is $r^{-1} = S_A \circ r$, which is still a
unitary cocentral map, i.e. in particular a coalgebra map. Hence,
the inverse of $S_A \circ r = r^{-1}$ in the group $CoZ^{1}(H, A)$
is $S_A \circ (S_A \circ r) = S_A^2 \circ r$. Thus $S_A^2 \circ r
= r$.
\end{remark}
\begin{examples} \exlabel{cocilaz}
1. If $H$ is cocommutative, then the group $CoZ^{1} (H, A)$
coincides with the group of all unitary coalgebra maps $r: H \to
A$ with the convolution product.

In particular, let $G$ and $G'$ be two groups, $H = k[G]$ and $A =
k[G']$ the corresponding group algebras. Then the group $CoZ^{1}
(k[G], k[G'])$ is isomorphic to the group of all unitary maps $r:
G \to G'$.

2. Let $H = A := H_4$ be the Sweedler's four dimensional Hopf
algebra. Then, by a routine computation proved in
\leref{morf_H_4k[C_n]} we can show that $CoZ^{1} (H_4, H_4)$ is
the trivial group with only one element, namely the trivial
unitary cocentral map $r: H_4 \to H_4$, $r (h) = \varepsilon (h)
1_H$, for all $h\in H_4$.

3. Consider now $A = H_{4}$ and $H = k[C_{n}]$, where $C_n$ is the
cyclic group of order $n$ generated by $c$. Let $r: k[C_{n}] \to
H_{4}$ be a unitary coalgebra map. Then \equref{0aa} is trivially
fulfilled as $H$ is cocomutative. Moreover, since $r$ is a
coalgebra map we have $r(c^{i}) \in \{1, \, g\}$ for all $i \in
\{1, 2, ..., n-1\}$ and $r(1) = 1$. Thus $CoZ^{1} (k[C_n], H_4)$
is the abelian group $C_2\times C_2 \times \cdots \times C_2$ of
order $2^{n-1}$.

On the other hand, we can prove that $CoZ^{1} (H_4, k[C_{n}])$ is
the group with only one element, namely the trivial unitary
cocentral map $r: H_4 \to k[C_{n}]$, $r (h) = \varepsilon (h)
1_{C_n}$, for all $h\in H_4$.

4. A general method of constructing unitary cocentral maps is the
following. Let $\psi: A \ot H \to A \ot H$ be a left $A$-linear
Hopf algebra isomorphism. Then
$$ r = r_{\psi} : H \to A, \quad
r(h) =  (({\rm Id}\ot \varepsilon_{H}) \circ \psi ) (1_A \bowtie
h)
$$
for all $h\in H$ is a unitary cocentral map. Furthermore, $a \,
r_{\psi} (h) = r_{\psi} (h) a$, for all $h\in H$ and $a\in A$.
Conversely, if $r: H \to A$ is a unitary cocentral map such that
${\rm Im}(r) \subseteq Z (A)$, then
$$
\psi = \psi_{r}: A \ot H \to A \ot H, \quad \psi (a \ot h) := a \,
r(h_{(1)}) \ot h_{(2)}
$$
is a left $A$-linear Hopf algebra isomorphism. For further details
we refer to \coref{clasforms1}.
\end{examples}

\begin{definition}\delabel{stabil}
Let $A$ be a Hopf algebra, $A \bowtie H$ and $A \bowtie' H'$ two
bicrossed products associated to two matched pairs of Hopf
algebras $(A, H, \triangleright, \triangleleft)$ and $(A, H',
\triangleright', \triangleleft')$. We say that a morphism of Hopf
algebras $\psi : A \bowtie H \to A \bowtie' H'$ \emph{stabilizes}
$A$ if the following diagram
\begin{equation}\eqlabel{D1}
\begin{CD}
A@>i_A>> A \bowtie H \\
@VV{\rm Id}_A V @VV\psi V\\
A@>i_A>> A \bowtie' H'
\end{CD}
\end{equation}
is commutative.
\end{definition}

A morphism of Hopf algebras $\psi : A \bowtie H \to A \bowtie' H'$
stabilizes $A$ if and only if $\psi$ is a morphism of Hopf
algebras and left $A$-modules, where a bicrossed product $A
\bowtie H$ is viewed as a left $A$-module via the restriction of
scalars through the canonical inclusion $i_A : A \to A \bowtie H$.
Such (iso)morphisms are fully described in the following:

\begin{theorem}\thlabel{clasif1}
Let $A$ be a Hopf algebra, $(A, H, \triangleright, \triangleleft)$
and $(A, H', \triangleright', \triangleleft')$ two matched pairs
of Hopf algebras. Then:

$(1)$ There exists a one-to-one correspondence between the set of
all Hopf algebra morphisms $\psi : A \bowtie H \to A \bowtie ' H'
$ that stabilize $A$ and the set of all pairs $(r, v)$, where $r:
H \rightarrow A$, $v: H \rightarrow H'$ are unitary coalgebra maps
satisfying the following compatibility conditions for any $a \in
A$, $g$, $h \in H$:
\begin{eqnarray}
r(h_{(1)}) \ot v(h_{(2)}) &{=}& r(h_{(2)}) \ot v(h_{(1)})\eqlabel{0}\\
r(hg) &{=}& r(h_{(1)})\bigl(v(h_{(2)}) \triangleright'
r(g)\bigl)\eqlabel{1}\\
v(hg) &{=}& \bigl(v(h) \triangleleft' r(g_{(1)})\bigl)
v(g_{(2)})\eqlabel{2}\\
h \triangleright a &{=}& r(h_{(1)}) \, \bigl(v(h_{(2)})
\triangleright' a_{(1)} \bigl) \, (S_A \circ r) (h_{(3)} \triangleleft a_{(2)})  \eqlabel{3}\\
v(h \triangleleft a) &{=}& v(h) \triangleleft' a \eqlabel{4}
\end{eqnarray}
Under the above bijection the morphism $\psi : A \bowtie H \to A
\bowtie ' H'$ corresponding to $(r, v)$ is given by:
\begin{equation}\eqlabel{psi2a}
\psi(a \bowtie h) = a \, r(h_{(1)}) \, \bowtie' \, v(h_{(2)})
\end{equation}
for all $a\in A$ and $h\in H$.

$(2)$ The left $A$-linear Hopf algebra morphism $\psi : A \bowtie
H \to A \bowtie ' H'$ given by \equref{psi2a} is an isomorphism if
and only if $v: H \to H'$ is bijective.\footnote{That is $v: (H,
\triangleleft) \rightarrow (H', \triangleleft')$ is an unitary
isomorphism of right $A$-module coalgebras.}
\end{theorem}

\begin{proof} $(1)$ We shall apply \thref{toatemorf} for $A' = A$.
Any morphism of Hopf algebras $\psi : A \bowtie H \to A \bowtie '
H'$ is given by \equref{morfbicros} for some unitary coalgebra
maps $(u, p, r, v)$ satisfying \equref{C1}-\equref{C8}. Now, such
a morphism $\psi : A \bowtie H \to A \bowtie ' H'$ makes the
diagram \equref{D1} commutative if and only if the map $\alpha : A
\to A \ot H'$ constructed in the proof of \thref{toatemorf} takes
the form $\alpha (a) = a \ot 1_{H'}$. This is equivalent to the
fact that $u: A \to A$ and $p: A \to H'$, constructed in the same
proof, are precisely the following: $u (a) = a$ and $p(a) =
\varepsilon_A (a) 1_{H'}$, for any $a\in A$. With these maps $u$
and $p$, the compatibility relations \equref{C1}-\equref{C8} are
reduced to \equref{0}-\equref{4}. For instance, \equref{C7} takes
the form
$$
(h_{(1)} \triangleright a_{(1)}) r(h_{(2)} \triangleleft a_{(2)})
= r(h_{(1)}) \bigl(v(h_{(2)}) \triangleright' a \bigl)
$$
which is equivalent to \equref{3}, as $r$ is invertible in the
convolution algebra $\Hom (H, A)$ with the inverse $S_A \circ r$.
We also note that \equref{C8} takes the easier form given by
\equref{4} which means precisely the fact that the unitary
coalgebra map $v: H \to H'$ is also a morphism of right
$A$-modules. Finally, the formula of $\psi$ given by
\equref{morfbicros} takes the simplified form \equref{psi2a}.

$(2)$ Assume first that $v : H \to H'$ is bijective with the
inverse $v^{-1}$. Applying ${\rm Id} \ot v^{-1}$ in \equref{0} we
obtain that
$$
r(h_{(1)}) \ot h_{(2)} = r(h_{(2)}) \ot h_{(1)}
$$
for all $h\in H$. Thus, $r$ is a unitary cocentral map. In
particular, it follows from \reref{involdisp} that $S_A^2 \circ r
= r$. Using this observation we can easily prove that the map
$$
\psi^{-1} : A \bowtie' H' \to A \bowtie  H, \quad \psi^{-1} (a
\bowtie' h') = a \, \bigl( S_A \circ r \circ v^{-1}\bigl)
(h'_{(1)}) \, \bowtie \, v^{-1}(h'_{(2)})
$$
for all $a\in A$ and $h'\in H'$ is the inverse of $\psi$.

Conversely, assume that $\psi : A \bowtie H \to A \bowtie ' H'$
given by \equref{psi2a} is bijective, that is an isomorphism of
Hopf algebras and left $A$-modules. Then, there exists a left
$A$-module Hopf algebra morphism $\varphi: A \bowtie ' H' \to A
\bowtie H$ such that $\psi \circ \varphi = Id_{A \bowtie ' H'}$
and $\varphi \circ \psi = Id_{A \bowtie H}$. It follows from the
first part of the theorem that there exist two unitary coalgebra
maps $q: H' \to A$ and $t: H' \to H$ satisfying the compatibility
conditions \equref{0}-\equref{4}, where the pair $(\triangleright,
\triangleleft)$ is interchanged with $(\triangleright',
\triangleleft')$, such that $\varphi$ is given by
$$
\varphi \, (a \bowtie' h') = a \, q(h_{(1)}') \, \bowtie \,
t(h_{(2)}')
$$
for all $a\in A$ and $h'\in H'$. Let $h\in H$. From $(\varphi
\circ \psi) (1_A \bowtie h)  = 1_A \bowtie h$ we obtain
$$
1_A \bowtie h =  r(h_{(1)}) (q \circ v)(h_{(2)}) \bowtie (t \circ
v) (h_{(3)})
$$
If we apply $\varepsilon_A$ on the first position in the above
equality we obtain $ t \circ v = {\rm Id}_H $. On the other hand,
let $h'\in H'$. From $(\psi \circ \varphi) (1_A \bowtie' h') = 1_A
\bowtie' h'$ we obtain
$$
1_A \bowtie' h' = q(h_{(1)}') (r \circ t)(h_{(2)}') \bowtie'
(v\circ t) (h_{(3)}')
$$
If we apply $\varepsilon_A$ on the first position we obtain $ v
\circ t = {\rm Id}_{H'}$. Hence, it follows that $v$ is bijective
and $v^{-1} = t$, as needed.
\end{proof}

Now, we shall fix two Hopf algebras $A$ and $H$. We define the
small category ${\mathcal M} {\mathcal P}(A, H)$ of all matched
pairs as follows: the objects of ${\mathcal M} {\mathcal P}(A, H)$
are the set of all pairs $(\triangleleft, \triangleright)$, such
that $(A, H, \triangleright, \triangleleft)$ is a matched pair of
Hopf algebras. A morphism $\psi : (\triangleleft, \triangleright)
\to (\triangleleft', \triangleright')$ in the category ${\mathcal
M} {\mathcal P}(A, H)$ is a Hopf algebra map $\psi : A \bowtie H
\to A \bowtie ' H' $ that stabilizes $A$. In order to classify up
to an isomorphism that stabilizes $A$ all Hopf algebras $E$ that
factorize through $A$ and $H$ we have to describe the skeleton of
the category ${\mathcal M} {\mathcal P}(A, H)$. This will be done
next.

\begin{definition}\delabel{coho}
Let $A$ and $H$ be two Hopf algebras. Two objects
$(\triangleright, \triangleleft)$ and $(\triangleright',
\triangleleft')$ of the category ${\mathcal M} {\mathcal P}(A, H)$
are called \emph{cohomologous} and we denote this by
$(\triangleright, \triangleleft) \approx (\triangleright',
\triangleleft')$ if there exists a pair of maps $(r, v)$ such
that:

$(1)$ $r: H \rightarrow A$ is a unitary cocentral map, $v: H
\rightarrow H$ is an unitary isomorphism of coalgebras satisfying
the following compatibilities for any $g$, $h \in H$:
\begin{eqnarray}
r(hg) &{=}& r(h_{(1)}) \, \bigl(v(h_{(2)}) \triangleright '
r(g)\bigl)\eqlabel{1aa}\\
v(hg) &{=}& \bigl(v(h) \triangleleft' r(g_{(1)})\bigl) \,
v(g_{(2)})\eqlabel{2aa}
\end{eqnarray}

$(2)$ The actions $(\triangleright, \triangleleft)$ are
implemented from $(\triangleright', \triangleleft')$ via $(r, v)$
as follows:
\begin{eqnarray}
h \triangleleft a &=& v^{-1} \bigl(v(h) \triangleleft' a \bigl) \eqlabel{3aaaaaa} \\
h \triangleright a &=& r(h_{(1)}) \, \bigl(v(h_{(2)})
\triangleright ' a_{(1)} \bigl) \, \bigl(S_A \circ r \circ
v^{-1}\bigl) ( v(h_{(3)}) \triangleleft ' a_{(2)}) \eqlabel{3aa}
\end{eqnarray}
for all $a \in A$ and $h \in H$, where $v^{-1}$ is the usual
inverse of the bijective map $v$.
\end{definition}

\begin{remark}\relabel{cobord} The condition
\equref{3aaaaaa} is equivalent to saying that $v : (H,
\triangleleft) \to (H, \triangleleft')$ is an isomorphism of right
$A$-module coalgebras. There exists a trivial object in ${\mathcal
M} {\mathcal P}(A, H)$, namely $(\triangleleft',
\triangleright')$, where $\triangleleft'$, $\triangleright'$ are
both the trivial actions. An object $(\triangleright,
\triangleleft)$ of ${\mathcal M} {\mathcal P}(A, H)$ is called a
\emph{coboundary} if $(\triangleright, \triangleleft)$ is
cohomologous with the trivial object $(\triangleleft',
\triangleright')$. Thus, if we write down the conditions from
\deref{coho} we can easily prove that an object $(\triangleright,
\triangleleft)$ of ${\mathcal M} {\mathcal P}(A, H)$ is a
coboundary if and only if the right action $\triangleleft$ is the
trivial action and the left action $\triangleright$ is implemented
by
$$
h \triangleright a = r(h_{(1)}) \, a \,  S_A \bigl(r
(h_{(2)})\bigl)
$$
for some unitary cocentral map $r: H \rightarrow A$ that is also a
morphism of Hopf algebras. This is in fact the necessary and
sufficient condition for a bicrossed product $A \bowtie H$ to be
isomorphic as left $A$-modules and Hopf algebras to the usual
tensor product $A \ot H$.
\end{remark}

The classification theorem now follows: the set of all isomorphism
types of bicrossed products $A\bowtie H$ which stabilize $A$ (i.e.
the skeleton of the category ${\mathcal M} {\mathcal P}(A, H)$) is
in bijection with a cohomologically type pointed set ${\mathcal
H}^{2} (A, H)$.

\begin{theorem} \textbf{(The classification of bicrossed
products)} \thlabel{clasth} Let $A$ and $H$ be two Hopf algebras.
Then $\approx$ is an equivalence relation on the set ${\mathcal M}
{\mathcal P}(A, H)$ and there exists an one-to-one correspondence
between the set of objects of the skeleton of the category
${\mathcal M} {\mathcal P}(A, H)$) and the pointed quotient set $
{\mathcal H}^{2} (A, H) := {\mathcal M} {\mathcal P}(A,
H)/\approx$.
\end{theorem}

\begin{proof} It follows from \thref{clasif1} that
$(\triangleright, \triangleleft) \approx (\triangleright',
\triangleleft')$ if and only if there exists a left $A$-linear
Hopf algebra isomorphism $\psi : A \bowtie H \to A \bowtie' H$,
where $A \bowtie H$ and  $A \bowtie' H$ are the bicrossed products
associated to the matched pairs $(A, H, \triangleright,
\triangleleft)$ and respectively $(A, H, \triangleright',
\triangleleft')$. The compatibility condition \equref{3aa} is
exactly \equref{3} taking into account \equref{3aaaaaa}. Thus,
$\approx$ is an equivalence relation on the set ${\mathcal M}
{\mathcal P}(A, H)$ and we are done.
\end{proof}

\thref{clasif1} has several applications. Three of them are given
bellow. First we shall apply it for infinite dimensional quantum
doubles. Let $D_{\lambda} (A, H)$ and $D_{\lambda'} (A, H')$ be
two generalized quantum doubles associated to two skew pairings
$\lambda$, $\lambda'$ as constructed in $(4)$ \exref{exempleban}.
We shall prove a necessary and sufficient condition for
$D_{\lambda} (A, H) \cong D_{\lambda'} (A, H')$, an isomorphism of
Hopf algebras and left $A$-modules.

\begin{corollary} \colabel{douadh}
Let $\lambda: H \ot A \rightarrow k$, $\lambda': H' \ot A
\rightarrow k$ be two skew pairings of Hopf algebras, $D_{\lambda}
(A, H)$ and $D_{\lambda'} (A, H')$ the generalized quantum
doubles. The following are equivalent:

$(1)$ There exists a left $A$-linear Hopf algebra isomorphism
$D_{\lambda} (A, H) \cong D_{\lambda'} (A, H')$;

$(2)$ There exists a pair of maps $(r, v)$, where $r: H
\rightarrow A$ is a unitary cocentral map, $v: H \rightarrow H'$
is an unitary isomorphism of coalgebras satisfying the following
four compatibility conditions:
\begin{eqnarray*}
&& r(hg) \, = \, r(h_{(1)}) \, r(g_{(2)}) \, \lambda' \bigl(
v(h_{(2)}),
\, r(g_{(3)}) \, S_A ( r(g_{(1)})) \bigl) \eqlabel{1aab}\\
&& v(hg) \, = \, v(h_{(2)}) v(g_{(2)}) \lambda' \bigl( S_{H'}
(v(h_{(1)}))\, v(h_{(3)}), \, r(g_{(1)})\bigl)\eqlabel{2aab}\\
&& v(h_{(2)}) \, \lambda \bigl(S_H (h_{(1)}) h_{(3)}, \,  a \bigl)
\, = \,  v(h_{(2)}) \, \lambda' \bigl(S_{H'} (v(h_{(1)})) \,
v(h_{(3)}),
\, a \bigl)\\
&& a_{(2)} r(h_{(2)}) \, \lambda \bigl(S_H (h_{(1)}), a_{(1)}
\bigl) \, \lambda \bigl(h_{(3)},  a_{(3)} \bigl) \, = \,  \\
&& = r(h_{(1)}) a_{(2)} \, \lambda' \bigl( v(h_{(3)}),  a_{(3)}
\bigl) \, \lambda' \bigl( S_{H'}(v(h_{(2)})),  a_{(1)} \bigl)
\eqlabel{3aab}
\end{eqnarray*}
for all $a \in A$, $g$, $h \in H$.
\end{corollary}

\begin{proof} First we note that $D_{\lambda} (A, H) = A \bowtie_{\lambda}
H$, where  the matched pair $(A, H, \triangleleft_{\lambda}, \,
\triangleright_{\lambda})$ is given in \equref{dhsk1} and
\equref{dhsk2}. Now, we are in position to apply \thref{clasif1}.
The four compatibility conditions from the second statement are
precisely \equref{1}-\equref{4} applied to the matched pairs
associated to $\lambda$ and $\lambda'$. In order to simplify these
compatibilities, the axioms of the skew pairings $\lambda$ and
$\lambda'$ are used as well as the fact that $\lambda (h, a) =
\lambda (S_H(h), S_A (a))$, for all $h\in H$ and $a\in A$
\cite[Lemma 1.4]{DT2}. The third compatibility above is
\equref{4}, which means exactly the fact that $v: H \rightarrow
H'$ is also a right $A$-module map.
\end{proof}

The next corollary provides necessary and sufficient conditions
for a bicrossed product $A \bowtie H$ to be isomorphic to a smash
product $A\#' H'$ such that the isomorphism stabilizes $A$.

\begin{corollary} \colabel{SchZ}
Let $(A, H, \triangleright, \triangleleft)$ be a matched pair of
Hopf algebras, $H'$ a Hopf algebra and $(A, \triangleright')$ a
left $H'$-module algebra and coalgebra satisfying the
compatibility condition \equref{smash1}. The following are
equivalent:

$(1)$ There exists a left $A$-linear Hopf algebra isomorphism $A
\bowtie H \cong A\#' H'$;

$(2)$ The right action $\triangleleft$ is the trivial action and
there exists a pair $(r, v)$, where  $r: H \rightarrow A$ is a
unitary cocentral map, $v: H \rightarrow H'$ is an isomorphism of
Hopf algebras such that the left action $\triangleright$ is
implemented by the formula:
\begin{equation}\eqlabel{implsmsh}
h \triangleright a = r (h_{(1)}) \, \bigl(v(h_{(2)})
\triangleright ' a \bigl) \, (S_A \circ r) (h_{(3)})
\end{equation}
and the following compatibility condition holds:
\begin{equation}\eqlabel{implsmshaa}
r(hg) = r(h_{(1)})\bigl(v(h_{(2)}) \triangleright' r(g)\bigl)
\end{equation}
for all $h$, $g\in H$ and $a\in A$.
\end{corollary}

\begin{proof}
We apply \thref{clasif1} for the matched pair $(A, H',
\triangleright', \triangleleft')$, where $\triangleleft'$ is the
trivial action. Using the fact that $v$ is bijective, we obtain
from the compatibility condition \equref{4} that the right action
$\triangleleft$ is also the trivial action. On the other hand
\equref{3} takes the equivalent form \equref{implsmsh} using that
$r$ is invertible in the convolution with the inverse $S_A \circ
r$.
\end{proof}

As a special case of \thref{clasif1} we have the following
interesting result:

\begin{corollary}\colabel{clasforms1}
Let $A$, $H$, $H'$ be three Hopf algebras. Then there exists a
bijection between the set of all left $A$-linear Hopf algebra
isomorphisms $\psi : A \ot H \to A \ot H' $ and the set of all
pairs $(r, v)$, where $v: H \rightarrow H'$ is an isomorphism of
Hopf algebras, $r: H \rightarrow A$ is a unitary cocentral map and
a morphism of Hopf algebras  with ${\rm Im} (r) \subset Z(A)$, the
center of $A$.

Under the above bijection the left $A$-linear Hopf algebra
isomorphism $\psi : A \ot H \to A \ot H'$ corresponding to $(r,
v)$ is given by:
\begin{equation}\eqlabel{psi2forms}
\psi(a \ot h) = a \, r(h_{(1)}) \, \ot \, v(h_{(2)})
\end{equation}
for all $a\in A$ and $h\in H$.
\end{corollary}

\begin{proof} We consider $\triangleleft$, $\triangleleft'$, $\triangleright$ and $\triangleright'$
to be all the trivial actions in \thref{clasif1}. The
compatibility conditions \equref{1}-\equref{3} implies in this
case that $r$ and $v$ are also algebra maps, while \equref{3}
becomes
$$
r(h_{(1)}) \, a \, (S_A \circ r ) (h_{(2)}) = \varepsilon_H(h) a
$$
for all $a\in A$, $h\in H$ which is equivalent to the centralizing
condition $r (h) a = a r(h)$, for all $a\in A$ and $h\in H$.
\end{proof}

\begin{corollary}\colabel{clasforms2}
Let $A$, $H$ and $H'$ be three Hopf algebras. Then $H$ and $H'$
are isomorphic as Hopf algebras if and only if there exists a left
$A$-linear Hopf algebra isomorphism $A \ot H \cong A \ot H'$.
\end{corollary}

\begin{proof}
If $f: H \to H'$ is an isomorphism of Hopf algebras, then ${\rm
Id}_A \ot f : A \ot H \to A \ot H'$ is a left $A$-linear Hopf
algebra isomorphism. Conversely, if $\psi : A \ot H \to A \ot H'$
is a left $A$-linear Hopf algebra isomorphism then the map $v: H
\to H'$ constructed in the proof of \coref{clasforms1} is in fact
an isomorphism of Hopf algebras.
\end{proof}

\section{Examples}\selabel{secexemple}
This section is devoted to the construction of some explicit
examples: for two given Hopf algebras $A$ and $H$ we will describe
and classify all Hopf algebras $E$ that factorize through $A$ and
$H$. There are three steps that we have to go through: first of
all we have to compute the set of all matched pairs between $A$
and $H$. Then we have to describe by generators and relations all
bicrossed products $A \bowtie H$ associated to these matched
pairs. Finally, using \thref{toatemorf}, we shall classify up to
an isomorphism the bicrossed products $A \bowtie H$. As an
application, the group $\Aut _{\rm Hopf} (A \bowtie H)$ of all
Hopf algebra automorphisms of a given bicrossed product is
computed.

For a Hopf algebra $H$, $G(H)$ is the set of group-like elements
of $H$ and for $g$, $h \in G(H)$ we denote by $P_{g, h} (H)$ the
set of all $(g, h)$-primitive elements, that is
$$
P_{g, \, h} (H) = \{ x\in H \, | \, \Delta_H(x) = x \ot g + h \ot
x\}
$$

The following result is very useful in computing all matched pairs
between $A$ and $H$:

\begin{lemma}\lelabel{primitive}
Let $(A, H, \triangleleft, \triangleright)$ be a matched pair of
Hopf algebras, $a$, $b \in G(A)$ and $g$, $h\in G(H)$. Then:

$(1)$ $g \triangleright a \in G(A)$ and $g\triangleleft a \in
G(H)$;

$(2)$ If $x \in P_{a, \, b}(A)$, then $g \triangleleft x \in
P_{g\triangleleft a, \, g\triangleleft b} (H)$ and $g
\triangleright x \in P_{g\triangleright a, \, g\triangleright b}
(A)$;

$(3)$ If $y \in P_{g, \, h}(H)$, then $y \triangleleft a \in
P_{g\triangleleft a, \, h\triangleleft a} (H)$ and $y
\triangleright a \in P_{g\triangleright a, \, h\triangleright a}
(A)$.

In particular, if $x$ is an $(1_A, b)$-primitive element of $A$,
then $g \triangleright x$ is an $(1_A, g\triangleright
b)$-primitive element of $A$ and $g \triangleleft x$ is an $(g,
g\triangleleft b)$-primitive element of $H$.
\end{lemma}

\begin{proof}
Straightforward: in fact for $(1)$-$(3)$ we only use the fact that
$\triangleleft$ and $\triangleright$ are coalgebra maps, i.e.
\equref{6} and \equref{8} hold. For the final statement we use the
normalizing conditions \equref{mp1} for $(\triangleleft,
\triangleright)$.
\end{proof}

From now on, $C_n$ will be the cyclic group of order $n$ generated
by $c$ and $k$ will be a field of characteristic $\neq 2$. Let $A
:= H_{4}$ be the Sweedler's $4$-dimensional Hopf algebra having
$\{1, \, g, \, x, \, gx \}$ as a basis subject to the relations:
$$
g^{2} = 1, \quad x^{2} = 0, \quad x g = -g x
$$
with the coalgebra structure and antipode given by:
$$
\Delta(g) = g \otimes g, \quad \Delta(x) = x \otimes 1 + g \otimes
x, \quad \Delta(gx) = gx \ot g + 1 \otimes gx
$$
$$\varepsilon(g) = 1, \quad \varepsilon(x) = 0, \quad S(g) = g,
\quad S(x) = -gx
$$

In order to compute the set of all matched pairs between $H_4$ and
$k[C_n]$ we need the following elementary result.

\begin{lemma}\lelabel{primitsw}
Let $k$ be a field of characteristic $\neq 2$,  $H_4$ the
Sweedler's $4$-dimensional Hopf algebra and $k[C_n]$ the group
algebra of $C_n$. Then:
$$
P_{c^i, \, c^j} (k[C_{n}]) = \{\lambda c^i - \lambda c^j \, | \,
\lambda \in k\}
$$
$$
G (H_4) = \{1, \, g \}, \quad P_{1, \, 1} (H_4) = \{0\}, \quad
P_{g, \, g} (H_4) = \{0\}
$$
$$
P_{1, \, g} (H_4) = \{ \alpha - \alpha g + \beta \, x \,\,\, | \,
\alpha, \beta \in k \}, \qquad  P_{g, \, 1} (H_4) = \{ \alpha -
\alpha g + \beta \, g x \, | \, \alpha, \beta \in k \}
$$
for all $i$, $j = 0, 1, \cdots, n-1$.
\end{lemma}

\begin{proof}
Everything is just a straightforward computation. For example, let
$z = \lambda_0 + \lambda_1c + \lambda_2 c^2 + \cdots +
\lambda_{n-1} c^{n-1} \in P_{c^i, \, c^j} (k[C_{n}])$, for some
$\lambda_j \in k$. Let $t \neq i$, $t \neq j$ and apply ${\rm Id}
\ot (c^t)^*$ in $\Delta (z) = z \ot c^i + c^j \ot z$ (where
$(c^t)^*_{t = 0, \cdots, n-1}$ is the dual basis of $({c^t})_{t =
0, \cdots, n-1}$ we obtain $\lambda_t c^t = \lambda_t c^j$, thus
$\lambda_t = 0$. Hence, $z = \lambda_i c^i + \lambda_j c^j$. Now
using $\Delta (z) = z \ot c^i + c^j \ot z$ we obtain that
$\lambda_i + \lambda_j = 0$, and thus $z = \lambda c^i - \lambda
c^j$, for some $\lambda \in k$.
\end{proof}

For a positive integer $n$ let $U_n (k) = \{ \omega \in k \, | \,
\omega^n = 1 \}$ be the cyclic group of $n$-th roots of unity in
$k$. The group $U_n (k)$ parameterizes the set of all matched
pairs $(H_4, k[C_n], \triangleleft, \triangleright)$.

\begin{proposition} \prlabel{mapex1}
Let $k$ be a field of characteristic $\neq 2$, $n$ a positive
integer and $C_n$ the cyclic group of order $n$. Then there exists
a bijective correspondence between the set of all matched pairs
$(H_4, k[C_n], \triangleleft, \triangleright)$ and $U_n (k)$ such
that the matched pair $(\triangleleft, \triangleright)$
corresponding to an $n$-th root of unity $\omega \in U_n (k)$ is
given by:
\begin{center}
\begin{tabular} {l | c  c  c  c  }
$\triangleleft$ & 1 & $g$ & $x$ & $gx$\\
\hline 1 & 1 & 1 & 0 & 0\\
c & c & c & 0 & 0 \\
\vdots & \vdots & \vdots & \vdots & \vdots\\
$c^k$ & $c^k$ & $c^k$ & 0 & 0\\
\vdots & \vdots & \vdots & \vdots & \vdots\\
$c^{n-1}$ & $c^{n-1}$ & $c^{n-1}$ & 0 & 0\\
\end{tabular}\, \qquad
\begin{tabular} {l | c  c  c  c  }
$\triangleright$ & 1 & $g$ & $x$ & $gx$\\
\hline 1 & 1 & g & x & gx\\
c & 1 & g & $\omega x$ & $\omega gx$ \\
\vdots & \vdots & \vdots & \vdots & \vdots\\
$c^k$ & 1 & g & $\omega^kx$ & $\omega^kgx$\\
\vdots & \vdots & \vdots & \vdots & \vdots\\
$c^{n-1}$ & 1 & g & $\omega^{n-1}x$ & $\omega^{n-1}gx$\\
\end{tabular}
\end{center}
\end{proposition}

\begin{proof}
Let $(H_4, k[C_n], \triangleleft, \triangleright)$ be a matched
pair. It follows from the normalizing conditions \equref{mp1} that
\begin{equation*}
c^i \triangleright 1 = 1, \quad 1 \triangleright z = z, \quad 1
\triangleleft z = \varepsilon (z) 1, \quad c^i \triangleleft 1 =
c^i
\end{equation*}
for all $i = 0, \cdots, n-1$ and $z \in H_4$. We prove now that $
c \triangleright g = g$. Indeed, it follows from \leref{primitive}
that $c \triangleright g$ is a group-like element in $H_4$. Thus,
using \leref{primitsw} we obtain that $c \triangleright g \in \{1,
\, g\}$. If $c \triangleright g = 1$ using the fact that
$\triangleright$ is a left action we obtain $c^i \triangleright g
= 1$, for any positive integer $i$. In particular, $g = 1
\triangleright g = c^n \triangleright g = 1$, contradiction. Thus,
$c \triangleright g = g$ and hence $c^i \triangleright g = g$, for
all $i = 1, \cdots, n-1$.

Now, $x$ is an $(1, g)$-primitive element of $H_4$. As
$c\triangleright g = g$ we obtain using \leref{primitive} that $c
\triangleright x$ is an $(1, g)$-primitive element of $H_4$. It
follows from \leref{primitsw} that $c\triangleright x = \beta -
\beta g + \omega \, x$, for some $\beta$, $\omega \in k$. We will
first prove that $\omega^n = 1$ and later on that $\beta = 0$.
Indeed, by induction, having in mind that $\triangleright$ is a
left action we obtain:
$$
c^i \triangleright x = \beta \sum_{j=0}^{i-1} \omega^j -
\beta(\sum_{j=0}^{i-1}\omega^j ) \, g + \omega^i \, x
$$
for any positive integer $i$. In particular,
$$
x = 1\triangleright x = c^n\triangleright x = \beta
\sum_{j=0}^{n-1} \omega^j - \beta(\sum_{j=0}^{n-1} \omega^j ) g +
\omega^n x
$$
thus $\omega^n = 1$ and $\beta \sum_{j=0}^{n-1} \omega^j = 0$.
Similarly, we obtain that for any $i = 1, \cdots, n-1$ we have:
$$
c^i \triangleright(gx) = \gamma \sum_{j=0}^{i-1}\zeta^j - \gamma
(\sum_{j=0}^{i-1}\zeta^j ) \, g + \zeta^i \, gx
$$
for some $\gamma$, $\zeta \in k$, such that $\zeta^n=1$ and
$\gamma \sum_{j=0}^{n-1}\zeta^j=0$. At the end of the proof we
will see that $\gamma = \beta = 0$ and $\zeta = \omega$.

Meanwhile, using what we already know about the left action
$\triangleright$, we shall prove that the other action
$\triangleleft$ of the matched pair $(H_4, k[C_n], \triangleleft,
\triangleright)$ is necessarily the trivial action of $H_4$ on
$k[C_n]$. First of all, using \leref{primitive}, we obtain that
$c^i \triangleleft g$ is a group-like element in $k[C_n]$, hence
$c^i\triangleleft g\in \{1, c, \cdots, c^{n-1}\}$. We claim now
that $c^i \triangleleft g = c^i$, for all $i = 1, \cdots, n-1$.
Indeed, suppose that there exist $i\neq j$ such that $c^i
\triangleleft g = c^j$. Then, the compatibility condition
\equref{mp4} from the definition of a matched pair applied for
$(c^i, x)$ takes the form:
$$
c^i\triangleleft x\ot c^i\triangleright1 + c^i\triangleleft g\ot
c^i\triangleright x = c^i\triangleleft 1\ot c^i\triangleright x +
c^i\triangleleft x\ot c^i\triangleright g
$$
which is equivalent to:
\begin{eqnarray*}
c^i \triangleleft x \ot 1 + \beta \, (\sum_{t=0}^{i-1}\omega^t)\,
c^j \ot 1 - \beta \, (\sum_{t=0}^{i-1}\omega^t) \, c^{j} \ot  g +
\omega^{i} \, c^{j} \ot x = \\
\beta \, (\sum_{t=0}^{i-1}\omega^t ) \, c^{i} \ot 1 - \beta \,
(\sum_{t=0}^{i-1}\omega^t ) \, c^{i} \ot g + \omega^{i} c^{i} \ot
x + c^{i} \triangleleft x \ot g
\end{eqnarray*}
Now observe that the coefficient of $c^j\ot x$ in the left hand
side of the equality is $\omega^i$ while the coefficient of the
same element in the right hand side of the equality is $0$ if
$i\neq j$. Since $\omega^n = 1$ we reached a contradiction. Thus
$c^i \triangleleft g = c^i$, for all $i = 1, \cdots, n-1$. Now,
$c^i$ is a group like element of $k[C_n]$ and $x \in P_{1, g}
(H_4)$. Thus, using \leref{primitive} we obtain that $c^i
\triangleleft x \in P_{c^i \triangleleft 1, \, c^i\triangleleft g}
(k[C_n]) = P_{c^i, \, c^i} (k[C_n])$. Hence, it follows from
\leref{primitsw} that $c^i \triangleleft x = 0$. A similar
argument shows that $c^i\triangleleft (gx) = 0$, for all $i = 1,
\cdots, n-1$.

Thus we have proved that $\triangleleft$ is the trivial action. In
this case the compatibility conditions from the definition of a
matched pair become simpler. More precisely the compatibility
\equref{mp2} becomes $h\triangleright(ab) = \left(h_{(1)}
\triangleright a\right) \left(h_{(2)}\triangleright b\right)$, for
all $h\in k[C_{n}]$, $a$ and $b\in H_{4}$, the compatibility
\equref{mp3} is trivially fulfilled while the compatibility
\equref{mp4} becomes $h_{(1)}\ot h_{(2)}\triangleright a =
h_{(2)}\ot h_{(1)}\triangleright a$, for all $h\in k[C_{n}]$,
$a\in H_{4}$ and is also automatically satisfied since $k[C_4]$ is
cocommutative.

Now, using the compatibility condition \equref{mp2} in its
simplified form $c\triangleright(g x) = (c\triangleright
g)(c\triangleright x)$, we obtain:
$$
\gamma - \gamma g + \zeta \, gx = g (\beta - \beta g + \omega x) =
-\beta + \beta g + \omega \, gx
$$
and therefore, $\gamma = -\beta$ and $\zeta = \omega$. Moreover,
since $0 = c\triangleright x^2 = (c\triangleright x)
(c\triangleright x)$, we find that
$$
0 = (\beta-\beta g+\omega x)^2 = 2\beta^2 - 2\beta^2 g + 2
\beta\omega x
$$
As ${\rm char}(k) \neq 2$ we obtain that $\beta = 0$ and, as a
consequence, that $c^i \triangleright x = \omega^i x$ and $c^i
\triangleright(gx) = \omega^i (gx)$, for all $i = 1, \cdots, n-1$
and hence the left action $\triangleright$ is also fully
described. We are left to verify the rest of the compatibilities,
and, since this is a routinely check, we omit it.
\end{proof}

We prove now that a Hopf algebra $E$ factorizes through $H_4$ and
$k[C_n]$ if and only if $E \cong H_{4n, \, \omega}$, where $H_{4n,
\, \omega}$ is a quantum group associated to any $\omega$ an
$n$-th root of unity as described bellow.

\begin{corollary}\colabel{primaclasi}
Let $k$ be a field of characteristic $\neq 2$ and $n$ a positive
integer. Then a Hopf algebra $E$ factorizes through $H_4$ and
$k[C_n]$ if and only if $E \cong H_{4n, \, \omega}$, for some
$\omega \in U_n (k)$, where we denote by $H_{4n, \, \omega}$ the
Hopf algebra generated by $g$, $x$ and $c$ subject to the
relations:
$$
g^{2} = c^n = 1, \quad x^{2} = 0, \quad x g = -g x, \quad c g = g
c, \quad c x = \omega \, x c
$$
with the coalgebra structure and antipode given by:
$$
\Delta(g) = g \ot g, \quad \Delta(c) = c \ot c, \quad \Delta(x) =
 x \ot 1 + g \ot x
$$
$$
\varepsilon(c) = \varepsilon(g) = 1, \quad \varepsilon(x) = 0,
\quad S(c) = c^{n-1}, \quad S(g) = g, \quad S(x) = -gx
$$
$H_{4n, \, \omega}$ is a pointed non-semisimple $4n$-dimensional
Hopf algebra.
\end{corollary}

\begin{proof}
It follows from \prref{mapex1} and \thref{carMaj}. The Hopf
algebra $H_{4n, \, \omega}$ is the bicrossed product $H_4 \bowtie
k[C_n]$ associated to the matched pair given in \prref{mapex1}.
Indeed, up to canonical identification, $H_4 \bowtie k[C_n]$ is
generated as an algebra by $ g = g \bowtie 1$, $x = x \bowtie 1$
and $c = 1\bowtie c$. Hence,
$$
c \, x = (1 \bowtie c) (x \bowtie 1) = c \triangleright x_{(1)}
\bowtie c \triangleleft x_{(2)} = \omega \, x \bowtie c = \omega
\, (x \bowtie 1) (1 \bowtie c) =  \omega \, x c
$$
$H_4 \bowtie k[C_n]$ is pointed by being generated by two
group-likes and a primitive element \cite[Lemma 1]{rk} and
non-semisimple as $H_4$ is non-semisimple.
\end{proof}

\begin{remark}
A $k$-basis in $H_{4n, \, \omega}$ is given by $\{ c^i, \, g c^i,
\, x c^i, g xc^i \, | \, i = 0, \cdots, n-1 \}$. We note that
$H_{4n, \, \omega}$ is the bicrossed product $H_4 \bowtie k[C_n]$
associated to the matched pair in which the right action
$\triangleleft : k[C_n] \ot H_4 \to k[C_n]$ is the trivial action.
Thus, $H_{4n, \, \omega}$ is in fact the semi-direct product $H_4
\# k[C_n]$ associated to the left action $\triangleright : k[C_n]
\ot H_4 \to H_4$ given in \prref{mapex1}. In particular, $H_4$ is
a normal Hopf subalgebra of $H_{4n, \, \omega}$.
\end{remark}

In order to classify these Hopf algebras $H_{4n, \, \omega}$ we
still need one more elementary lemma:

\begin{lemma}\lelabel{morf_H_4k[C_n]}
Let $k$ be a field of characteristic $\neq 2$,  $H_4$ the
Sweedler's $4$-dimensional Hopf algebra and $k[C_n]$ the group
algebra of $C_n$. Then:

$(1)$ $u : H_4 \rightarrow H_4$ is a unitary coalgebra map if and
only if $u$ is the trivial morphism $u(h) = \varepsilon (h) 1$,
for all $h\in H_4$, or there exist $\alpha$, $\beta$, $\gamma$,
$\delta \in k$ such that $u(1) = 1$, $u(g) = g$, $u(x) = \alpha -
\alpha \, g + \beta \, x$, and $u(gx) = \gamma - \gamma \, g +
\delta \, gx$. In particular, $CoZ^{1} (H_4, H_4)$ is the trivial
group with only one element.

$(2)$ $u : H_4 \to H_4$ is a Hopf algebra morphism if and only if
$u$ is the trivial morphism $u(h) = \varepsilon(h)1$, for all $h
\in H_4$, or there exists $\beta \in k$ such that $u(1) = 1$,
$u(g) = g$, $u(x) = \beta \, x$, and $u(gx) = \beta \, gx$. In
particular, $\Aut_{\rm Hopf} (H_4) \cong k^*$.

$(3)$ Assume that $n$ is odd. Then:
\begin{enumerate}
\item[(3a)] $p : H_4 \rightarrow k[C_n]$ is a morphism of Hopf
algebras if and only if $p$ is the trivial morphism: $p(h) =
\varepsilon(h) 1$, for all $h\in H_4$.

\item[(3b)] $r: k[C_n] \to H_4$ is a morphism of Hopf algebras if
and only if $r$ is the trivial morphism: $r (c^i) = 1$, for all $i
= 0, \cdots, n-1$.
\end{enumerate}
$(4)$ Assume that $n = 2m$ is even. Then:
\begin{enumerate}
\item[(4a)] $p : H_4 \rightarrow k[C_n]$ is a morphism of Hopf
algebras if and only if $p$ is the trivial morphism or $p$ is
given by: $p(1) = 1$, $p(g) = c^m$, $p(x) = p(gx) = 0$.

\item[(4b)] $r: k[C_n] \to H_4$ is a morphism of Hopf algebras if
and only if $r$ is the trivial morphism or $r$ is given by $r
(c^i) = g^i$, for all $i = 0, \cdots, n-1$.
\end{enumerate}
\end{lemma}

\begin{proof}
$(1)$ Let $u : H_4 \rightarrow H_4$ be a unitary coalgebra map.
Then $u(g) \in G(H_4) = \{1, g\}$ and $u(x) \in P_{1, \,
u(g)}(H_4)$ respectively $u(gx) \in P_{u(g), \, 1}(H_4)$. By
applying \leref{primitsw} we arrive at the desired conclusion: if
$u(g) = 1$, then $u$ is the trivial morphism, while if $u(g) = g$,
then $u$ has the second form. A little computation shows that
among all these morphisms only the trivial one satisfies the
cocycle condition \equref{0aa}.

$(2)$ It follows from $(1)$. If $u: H_4\to H_4$ is the non-trivial
morphism, then considering the relations on the generators of
$H_4$ we obtain that $u$ is also an algebra map if and only if
$\alpha = \gamma = 0$ and $\delta = \beta$. The final assertion is
obvious.

$(3)$ and $(4)$ Let now $p : H_4 \rightarrow k[C_n]$ be a Hopf
algebra map. Then, $p(g)$ is a group-like element in $k[C_n]$,
i.e. $p(g) = c^i$, for some $i \in \{0,1,...,n-1\}$. It follows
from
$$
1 = p(g^2) = p(g)^2 = c^{2i}
$$
that $n \mid 2i$. If $n$ is odd, then $n \mid i$, hence $i = 0$
i.e. $p(g) = 1$. If $n = 2m$ is even, then $m \mid i$, hence $i =
0$ or $i =  m$. Thus, $p(g) = 1$ or $p(g) = c^m$. On the other
hand, $x$ is an $(1, g)$-primitive element of $H_4$, hence $p(x)
\in P_{1, \, p(g)}(k[C_n])$. Using \leref{primitsw} we obtain that
$p(x) = 0$ if $p(g) = 1$, and $p(x) = \lambda - \lambda c^m$ if
$p(g) = c^m$. In the last case we have:
$$
0 = p(x^2) = p(x)^2 = \left(\lambda -\lambda c^m \right)^2 =
2\lambda^2 - 2\lambda^2 c^m
$$
therefore $\lambda = 0$ i.e. $p(x) = 0$. A similar discussion
describes Hopf algebra maps $r: k[C_n] \to H_4$ and we are done.
\end{proof}

We shall give now the necessary and sufficient conditions for two
Hopf algebras $H_{4n, \, \omega}$ and $H_{4n, \, \omega'}$ to be
isomorphic. The classification \thref{2.2} bellow depends on the
structure of $U_n (k)$. We denote by $\nu (n) = |U_n (k)|$, the
order of the cyclic group $U_n (k)$ and we shall fix $\xi$ a
generator of $U_n (k)$. Let $\omega = \xi^l$, $\omega' = \xi^t$ be
two arbitrary $n$-th roots of unity, for some positive integers
$l$, $t \in \{\, 0, \cdots, \nu (n) -1 \, \}$.

\begin{theorem}\thlabel{2.2}
Let $k$ be a field of characteristic $\neq 2$, $n$ a positive
integer, $\xi$ a generator of the group $U_n (k)$ of order $\nu
(n)$ and $l$, $t \in \{\, 0, \cdots, \nu (n) -1 \, \}$. Then:

$(1)$ Assume that one of the positive integers $n$ or $\nu (n)$ is
odd. Then the Hopf algebras $H_{4n, \, \xi^l}$ and $H_{4n, \,
\xi^t}$ are isomorphic if and only if there exists $s \in \{0, 1,
\cdots, n-1\}$ such that ${\rm gcd} \, (s, n) = 1$ and $\nu (n)
\mid l- ts$.

$(2)$ Assume that $n$ and $\nu (n)$ are both even. Then the Hopf
algebras $H_{4n, \, \xi^l}$ and $H_{4n, \, \xi^t}$ are isomorphic
if and only if there exists $s \in \{0, 1, \cdots, n-1\}$ such
that ${\rm gcd} \, (s, n) = 1$ and one of the following conditions
hold: $\nu (n) \mid l- ts$ or $2(l - ts) = \nu(n) \, q $, for some
odd integer $q$.
\end{theorem}

\begin{proof} We shall prove more. In fact we will describe the set of all
possible Hopf algebra isomorphisms between $H_{4n, \, \xi^l}$ and
$H_{4n, \, \xi^t}$. We denote by $H_{4n, \, \xi^l} := H_4 \#^l
k[C_n]$ (resp. $H_{4n, \, \xi^t} := H_4 \#^t k[C_n]$) the
semidirect product implemented by the left action $\triangleright
: k[C_n] \ot H_4 \to H_4$, $c \triangleright x = \xi^l x$ (resp.
$c \triangleright' x = \xi^t x$) from \prref{mapex1}. We recall
from \coref{endosmh} that $\psi : H_{4n, \, \xi^l} \to H_{4n, \,
\xi^t}$ is an isomorphism of Hopf algebras if and only if there
exists two unitary coalgebra maps $u: H_{4} \rightarrow H_{4}$,
$r: k[C_{n}] \rightarrow H_{4}$ and two morphisms of Hopf algebras
$p: H_{4} \rightarrow k[C_{n}]$, $v : k[C_{n}] \rightarrow
k[C_{n}]$ such that for any $a$, $b \in H_4$, $g$, $h \in k[C_n]$
we have
\begin{eqnarray}
u(a_{(1)}) \ot p(a_{(2)}) &{=}& u(a_{(2)}) \ot p(a_{(1)})\eqlabel{C1ab10}\\
r(h_{(1)}) \ot v(h_{(2)}) &{=}& r(h_{(2)}) \ot v(h_{(1)})\eqlabel{C2ab10}\\
u(ab) &{=}& u(a_{(1)}) \, \bigl( p (a_{(2)}) \triangleright ' u(b) \bigl)\eqlabel{C3ab10}\\
r(hg) &{=}& r(h_{(1)}) \, \bigl(v(h_{(2)}) \triangleright ' r(g)\bigl)\eqlabel{C5ab10}\\
r(h_{(1)}) \bigl(v(h_{(2)}) \triangleright' u(b) \bigl) &{=}& u
(h_{(1)} \triangleright b_{(1)}) \, \Bigl( p (h_{(2)}
\triangleright b_{(2)}) \triangleright ' r(h_{(3)})
\Bigl) \eqlabel{C7ab10}\\
v(h) \, p(b) &{=}& p (h_{(1)} \triangleright b ) \, v (h_{(2)})
\eqlabel{C8ab10}
\end{eqnarray}
and $\psi = \psi_{(u, r, p, v)} : H_{4n, \, \xi^l} \to H_{4n, \,
\xi^t}$ is given by:
\begin{equation}\eqlabel{morfbicrossmab10}
\psi(a \# h) = u(a_{(1)}) \, \bigl( p(a_{(2)}) \triangleright'
r(h_{(1)}) \bigl) \,\, \# \,  p(a_{(3)}) \, v(h_{(2)})
\end{equation}
for all $a \in H_4$ and $h\in k[C_n]$. In what follows we describe
completely all quadruples $(u, r, p, v)$ which satisfy the
compatibility conditions \equref{C1ab10}-\equref{C8ab10} such that
$\psi = \psi_{(u, r, p, v)}$ given by \equref{morfbicrossmab10} is
bijective. First, we should notice that \equref{C2ab10} holds
trivially as $k[C_{n}]$ is cocommutative. Also any morphism of
Hopf algebras $v: k[C_n] \to k[C_n]$ is given by
\begin{equation}\eqlabel{v-ul}
v: k[C_n] \to k[C_n], \quad v (c) = c^s
\end{equation}
for some $s = 0, \cdots, n-1$. For future use, we note that such a
morphism $v$ is bijective if and only if $(s, n) = 1$.

Next we shall prove simultaneously that any Hopf algebra map $p:
H_{4} \rightarrow k[C_{n}]$ of a such quadruple $(u, r, p, v)$ is
the trivial morphism
\begin{equation}\eqlabel{p-ul}
p: H_4 \to k[C_n], \quad p (z) = \varepsilon (z) 1
\end{equation}
for any $z\in H_4$ and any unitary coalgebra morphism $u: H_{4}
\rightarrow H_{4}$ of a such quadruple $(u, r, p, v)$ is given by
\begin{equation}\eqlabel{u-ul}
u: H_{4} \rightarrow H_{4}, \qquad u(1) = 1, \quad u (g) = g,
\quad u(x) = \gamma \, x, \quad u(gx) = \gamma \, gx
\end{equation}
for some non-zero scalar $\gamma \in k^*$. Indeed, it follows from
$(3a)$ and $(4a)$ of \leref{morf_H_4k[C_n]} that $p(x) = 0$ and
$p(gx) = 0$. Now, using the normalizing conditions \equref{mp1},
the fact that $r$, $v$ are unitary maps and $p$ is a coalgebra map
we obtain from \equref{morfbicrossmab10} that
$$
\psi ( x \# 1) = u(x_{(1)}) \# p(x_{(2)}) = u(x) \# 1 + u(g) \#
p(x) = u(x) \# 1
$$
As $\psi$ has to be an isomorphism and $x$ is an element of the
basis, we obtain that $u(x) \neq 0$. It follows from $(1)$ of
\leref{morf_H_4k[C_n]} that $u (x) = \alpha - \alpha \, g + \gamma
\, x$, for some $\alpha$, $\gamma \in k$. Now, by applying
\equref{C1ab10} for $a = x$, we obtain:
$$
(\alpha - \alpha \, g + \gamma \, x) \ot 1 = (\alpha - \alpha \, g
+ \gamma \, x) \ot p(g)
$$
As $u(x) \neq 0$, we must have $p(g) = 1$ and therefore $p: H_{4}
\rightarrow k[C_{n}]$ is the trivial morphism. Condition
\equref{C3ab10} is then equivalent to $u: H_4 \to H_4$ being an
algebra map, i.e. $u$ is a morphism of Hopf algebras. Using $(2)$
of \leref{morf_H_4k[C_n]} we obtain that $u$ is given by
\equref{u-ul}, where $\gamma \in k^{*}$, since $u(x) \neq 0$.
Thus, we fully described the maps $u$, $v$ and $p$ in the
quadruple $(u, r, p, v)$. Furthermore, since $p$ is the trivial
morphism the compatibility conditions \equref{C1ab10} and
\equref{C8ab10} are trivially fulfilled.

We focus now on the unitary morphism of coalgebras $r: k[C_n] \to
H_4$ in a such quadruple. As $r$ is a coalgebra map we have that
$r (c^i) \in G(H_4) = \{1, g\}$, for all $i = 0, \cdots, n-1$.
Since $c^{i} \triangleright' z = z$, for all $i = 0, \cdots, n-1$
and $z \in  \{1, g\}$, it follows that \equref{C5ab10} is
equivalent to $r: k[C_{n}] \rightarrow H_{4}$ being an algebra
map, that is $r : k[C_n] \to H_4$ is a Hopf algebra map. In
conclusion, taking into account that $p$ is the trivial morphism,
we proved so far that any potential isomorphism $\psi: H_{4n, \,
\xi^l} \to H_{4n, \, \xi^t}$  is defined by the formula
\begin{equation}\eqlabel{nouppsi}
\psi: H_4 \, \#^l \, k[C_n] \to H_4 \, \#^t \, k[C_n], \qquad
\psi(a \# h) = u(a) \, r(h_{(1)}) \, \# \, v(h_{(2)})
\end{equation}
for all $a \in H_4$ and $h\in k[C_n]$, where $u: H_4 \to H_4$ is
the isomorphism of Hopf algebras given by \equref{u-ul}, $v:
k[C_n] \to k[C_n]$ is the Hopf algebra map given by \equref{v-ul}
and $r: k[C_n] \to H_4$ is a morphism of Hopf algebras satisfying
the compatibility condition \equref{C7ab10} in its simplified
form, namely
\begin{equation} \eqlabel{C7'ab10}
r(h_{(1)}) \bigl( v(h_{(2)}) \triangleright' u(b) \bigl) = u
(h_{(1)} \triangleright b) \, r(h_{(2)})
\end{equation}
for all $h\in k[C_n]$ and $b\in H_4$.

Finally, it remains to describe when there exists a Hopf algebra
map $r: k[C_n] \to H_4$ satisfying the compatibility condition
\equref{C7'ab10} such that $\psi$ given by \equref{nouppsi} is
bijective.

According to \leref{morf_H_4k[C_n]} we distinguish two cases.
Suppose first that $n$ is odd. Then using $(3b)$ of
\leref{morf_H_4k[C_n]} we obtain that $r: k[C_n] \to H_4$ is also
the trivial morphism namely, $r(c^{i}) = \varepsilon (c^i) 1 = 1$,
for all $i \in 0, 1, \cdots, n-1$. Hence, $\psi$ given by
\equref{nouppsi} takes the simplified form $\psi(a \# h) = u(a) \#
v(h)$ and $\psi$ is an isomorphism if and only if $v$ is
bijective, i.e. if and only if $(s, n) = 1$. Moreover, the
compatibility condition \equref{C7'ab10} becomes:
\begin{equation}\eqlabel{last}
v(h) \triangleright' u(a) = u(h \triangleright a)
\end{equation}
for all $a \in H_{4}$ and $h \in k[C_{n}]$. Applying the above
compatibility for $a = x$ and $h = c$ we obtain $\xi^{l - s t} =
1$, thus $\nu (n) \mid (l-st)$. Moreover, it is easy to see that
if $\nu (n) \mid (l-st)$, then \equref{last} holds for any $a \in
H_{4}$ and $h \in k[C_{n}]$.

Finally, suppose now that $n$ is even and use $(4b)$ of
\leref{morf_H_4k[C_n]}. If $r$ is the trivial morphism then the
proof follows exactly as in the above odd case. Assume now that
$r(c) = g$. Hence, $\psi$ is given by $\psi(a \# h) = u(a)r(h) \#
v(h)$, for all $a \in H_{4}$ and $h \in \{1, c, \cdots,
c^{n-1}\}$. It is easy to see that $\psi$ is an isomorphism if and
only if $v$ is an isomorphism. i.e. if and only if $(s, n) = 1$.
Now, we observe that the compatibility condition \equref{C7'ab10}
is equivalent to $\xi^{l-ts} = -1$. Indeed, \equref{C7'ab10}
applied for $a = x$ and $h = c$ gives precisely
 $\xi^{l-ts} = -1$. Conversely, if \equref{C7'ab10} holds for $a =
x$ and $h = c$, then it is straightforward to see that it is
fulfilled for any $a \in H_{4}$ and any $h \in k[C_{n}]$.

We shall prove now that $\xi^{l-ts} = -1$ if and only if $ \nu
(n)$ is even and $2(l - ts) = \nu(n) \, q$, for some odd integer
$q$ and this finishes the proof.

Indeed, since ${\rm Char}(k) \neq 2$ and $\xi^{\nu(n)} = 1$, the
equation $\xi^{l-ts} = -1$ is possible only if $ \nu (n)$ is even.
Consider $\nu(n) = 2m$, for a positive integer $m$. As $\xi^{2m} =
1$, we obtain $\xi^{m} = -1$, otherwise using that $2m$ is the
order of $\xi$ we will obtain $2m | m$, contradiction. Assume
first that $\xi^{l-ts} = -1$. Then, $\xi^{2(l-ts)} = 1$ and hence
$\nu(n)\, | \, 2(l-ts)$, that is $l-ts = m q$, for some integer
$q$. Thus, $ - 1 = \xi^{l-ts} = (\xi^{m})^{q} = (-1)^{q}$, hence
$q$ is odd. Conversely, is straightforward.
\end{proof}

\begin{corollary}\colabel{2.3}
Let $k$ be a field of characteristic $\neq 2$, $n$ a positive
integer and $\xi$ a generator of the group $U_n (k)$. Then,
$H_{4n, \, \xi^{l}}$ is isomorphic to $H_{4n, \, \xi}$, for any
positive integer $l \in \{1, \cdots, \nu(n) - 1 \}$ such that
${\rm gcd} \, (l, n) = 1$.

In particular, if $n = p$ is a prime odd number, then a Hopf
algebra factorizes through $H_{4}$ and $k[C_{p}]$ if and only if
it is isomorphic to $H_{4n, 1}$ or $H_{4n, \xi}$, where $\xi$ is a
generator of $U_p(k)$.
\end{corollary}

\begin{proof}
We apply \thref{2.2} for $t = 1$ by taking $s := l$ in its
statement.
\end{proof}

In what follows we continue our investigation in order to indicate
precisely the number of all types of isomorphisms of Hopf algebras
which factorize through  $H_4$ and $k[C_n]$. For this we need the
following lemma in the proof of which we use Dirichlet's theorem
on primes in an arithmetical progression.

\begin{lemma}\lelabel{lemaelem}
Let $n$ and $m$ be two positive integers such that $ m \mid n $
and $ \varphi : \mathbb{Z}_n \rightarrow \mathbb{Z}_m $ the
canonical projection $ \varphi(a + n\mathbb{Z}) = a + m\mathbb{Z}
$, for all $a \in \mathbb{Z}$. Then $\varphi \left(
U(\mathbb{Z}_n) \right) = U(\mathbb{Z}_m)$.
\end{lemma}

\begin{proof}
Consider $a + m\mathbb{Z} \in U(\mathbb{Z}_m)$. Thus, $\gcd(a,m) =
1$. The Dirichlet's theorem \cite[Theorem 8]{Chan} ensure the fact
that in the set $\{a + km \vert k \in \mathbb{Z} \}$ there exists
infinitely many primes. In particular, there exists a prime number
$p \in \{a + km \vert k \in \mathbb{Z} \}$ such that $p \nmid n$.
As $\varphi (p + n\mathbb{Z}) = a + m\mathbb{Z}$, and $p +
n\mathbb{Z} \in U(\mathbb{Z}_n)$, we deduce that $U(\mathbb{Z}_m)
\subseteq \varphi \left( U(\mathbb{Z}_n) \right)$. Obviously,
$\varphi \left( U(\mathbb{Z}_n) \right) \subseteq
U(\mathbb{Z}_m)$, hence our claim.
\end{proof}

We shall describe and count the set of types of Hopf algebras that
factorize through $H_4$ and $k[C_n]$ .

\begin{theorem}\thlabel{clasnoudir}
Let $k$ be a field of characteristic $\neq 2$, $n$ a positive
integer, $\xi$ a generator of $U_n(k)$ and let $\nu(n) =
p_1^{\alpha_1} \cdots p_r^{\alpha_r}$ be the prime decomposition
of $\nu(n)$. Then:

$(1)$ There exists an isomorphism of Hopf algebras $H_{4n, \,
\xi^t} \simeq H_{4n, \, \xi^{\gcd(t, \, \nu(n))}}$, for all $t =
0, \cdots ,\nu(n)-1$. In particular, $H_{4n, \, \xi^i} \simeq
H_{4n, \, \xi^{n - i}}$, for any $i = 0, \cdots, \nu(n) - 1$.

$(2)$ Assume that $\nu(n)$ is odd. Then the set of types of Hopf
algebras that factorize through $H_4$ and $k[C_n]$ is in bijection
with the set of all Hopf algebras $H_{4n, \, \xi^{d}}$, where $d$
running over all positive divisors of $\nu(n)$. In particular, the
number of types of such Hopf algebras is
$$
(\alpha_1 + 1)(\alpha_2 + 1) \cdots (\alpha_r + 1)
$$

$(3)$ Assume that $\nu(n) = 2^{\alpha_1} p_2^{\alpha_2} \cdots
p_r^{\alpha_r}$ is even. Then the set of types of Hopf algebras
that factorize through $H_4$ and $k[C_n]$ is in bijection with the
set of all Hopf algebras $H_{4n, \, \xi^{d}}$, where $d$ is
running over all positive divisors of
$\displaystyle\frac{\nu(n)}{2}$. In particular, the number of
types of such Hopf algebras is
$$
\alpha_1(\alpha_2 + 1) \cdots (\alpha_r + 1)
$$
\end{theorem}

\begin{proof} $(1)$ Let $d = \gcd(t, \, \nu(n))$. Consider two positive
integers $a$ and $m$ such that $t = d a$ and $\nu(n) = d m$. Since
$\gcd(a, \, m) = 1$ and $m \mid \nu(n) \mid n$, we obtain from
\leref{lemaelem} that there exists $s \in \{0,1,...,n - 1\}$ such
that $\gcd(s,n) = 1$ and $m \mid a - s$. Multiplying the last
relation by $d$, we obtain $\nu(n) \mid t - ds$. This, together
with \thref{2.2}, proves that $H_{4n, \, \xi^t} \simeq H_{4n, \,
\xi^d}$. Since $\nu(n) | n$, the last statement follows from the
first part using that $\gcd (i, \nu(n)) = \gcd (n-i, \nu(n))$, for
any $i = 1, \cdots, \nu(n) - 1$.

$(2)$ In the first part we have proved that any Hopf algebra
$H_{4n, \, \xi^t}$ is isomorphic to $H_{4n, \, \xi^d}$, for some
divisor $d$ of $\nu(n)$. We will show now that if $d_1$ and $d_2$
are two distinct positive divisors of $\nu(n)$ then
$H_{4n,\xi^{d_1}}$ and $H_{4n,\xi^{d_2}}$ are not isomorphic. This
will prove our claim.

Let therefore $d_1$ and $d_2$ be two distinct positive divisors of
$\nu(n)$. Suppose that $H_{4n,\xi^{d_1}} \simeq H_{4n,\xi^{d_2}}$.
By \thref{2.2} there exists then $s \in \{0,1,...,n-1\}$ such that
$\gcd(s,n)=1$ and $\nu(n) \mid d_1 - sd_2$. Since $d_1$ and $d_2$
are distinct, they differ by the exponent of a prime number, say
$p$. Let $\alpha$ and $\beta$ be the respective exponents of $d_1$
and $d_2$. We do not restrict the generality by assuming that
$\alpha
> \beta$. Since
$$
\frac{\nu(n)}{p^{\beta}} \mid \frac{d_1}{p^{\beta}} -
s\frac{d_2}{p^{\beta}}
$$
$p \mid \displaystyle\frac{\nu(n)}{p^{\beta}}$, $p \mid
\displaystyle\frac{d_1}{p^{\beta}}$, and $p \nmid
s\displaystyle\frac{d_2}{p^{\beta}}$, we have arrived at the
desired contradiction.

$(3)$ We first prove that if $d_1$ and $d_2$ are two distinct
positive divisors of $\displaystyle\frac{\nu(n)}{2}$ then
$H_{4n,\xi^{d_1}} \ncong H_{4n,\xi^{d_2}}$. Then we prove that if
$d_1$ is a positive divisor of $\nu(n)$ then there exists $d_2$ a
positive divisor of $\displaystyle\frac{\nu(n)}{2}$ such that
$H_{4n,\xi^{d_1}} \cong H_{4n,\xi^{d_2}}$. This, together with
assertion $(1)$ will finish the proof.

Suppose $d_1$ and $d_2$ are two distinct positive divisors of
$\displaystyle\frac{\nu(n)}{2}$. If $s \in \{0,1,...,n-1\}$ and
$\gcd(s,n) = 1$, then $\displaystyle\frac{\nu(n)}{2} \nmid d_1 -
d_2s$. Indeed, $d_1$ and $d_2$ being distinct they differ by the
exponent of a prime number, $p$. Let $\alpha$ and $\beta$ be the
respective exponents of $d_1$ and $d_2$. We may assume, without
loss of generality that $\alpha > \beta$. Since $p \mid
\displaystyle\frac{\nu(n)}{2p^{\beta}}$, $p \mid
\displaystyle\frac{d_1}{p^{\beta}}$, and $p \nmid
s\displaystyle\frac{d_2}{p^{\beta}}$ (recall that $\nu(n) \mid n$
and $\gcd(s,n) = 1$) we cannot have
$$
\frac{\nu(n)}{2p^{\beta}} \mid \frac{d_1}{p^{\beta}} -
s\frac{d_2}{p^{\beta}}
$$
hence, neither $\displaystyle\frac{\nu(n)}{2} \mid d_1 - d_2s$.
This being the case, we deduce from \thref{2.2} that
$H_{4n,\xi^{d_1}}$ and $H_{4n,\xi^{d_2}}$ are not isomorphic.

For the second claim, consider $d_1$ a positive divisor of
$\nu(n)$. If $d_1 = \nu(n)$, then $H_{4n,\xi^{d_1}} = H_{4n,\xi^0}
\cong H_{4n,\xi^{\frac{\nu(n)}{2}}}$, as it results from
\thref{2.2} observing that:
$$2 \left(
\displaystyle\frac{\nu(n)}{2} - 1 \cdot 0 \right) = \nu(n) \cdot
1$$
Thus $d_2 = \displaystyle\frac{\nu(n)}{2}$ in this case. If
$d_1 \mid \displaystyle\frac{\nu(n)}{2}$, we take $d_2 = d_1$. If
$d_1 \neq \nu(n)$ and $d_1 \nmid \displaystyle\frac{\nu(n)}{2}$,
we take $d_2 = \displaystyle\frac{d_1}{2}$. Indeed, $2 \nmid
\displaystyle\frac{\nu(n)}{d_1}$, hence $\nu(n) = 2^{\alpha} u$
and $d_1 = 2^{\alpha} v$, for some positive integer $\alpha$ and
odd integers $u$ and $v$ such that $v \mid u$. Let $-q$ be the
product of all prime divisors of $n$ that do not divide
$2\displaystyle\frac{u}{v}$, and $s = 2 -
\displaystyle\frac{u}{v}q$. Then $q$ is an odd integer,
$-\displaystyle\frac{u}{v}q \leq \displaystyle\frac{n}{2}$, and
$\gcd(s,n) = 1$. Also,
$$
s = 2 - \frac{u}{v}q \leq 2 + \frac{n}{2} \leq n - 1
$$
as soon as $n \geq 6$. Multiplying $s = 2 -
\displaystyle\frac{u}{v}q$ by $d_2 = \displaystyle\frac{d_1}{2}$
we find that $2(d_1 - sd_2) = \nu(n)q$. Therefore, when $n \geq
6$, we have $H_{4n,\xi^{d_1}} \cong H_{4n,\xi^{d_2}}$, by virtue
of \thref{2.2} . If $n < 6$ then $\nu(n) = 2$ or $\nu(n) = 4$,
cases in which there is nothing more to prove.
\end{proof}

\begin{example}
A straightforward computation based on \thref{clasnoudir} shows
the following: if $\nu(n) = 2$ then there is only one type of
isomorphism, namely the tensor product while if $ \nu(n) \in \{3,
4, 5, 6, 7\}$ there are two types of isomorphisms of bicrossed
products between $H_{4}$ and $k[C_{n}]$ as follwos: $H_{4n, \, 1}$
and $H_{4n, \, \xi}$. If $\nu(n) = 8$ or $\nu(n) = 9$ then there
are three types of isomorphisms of bicrossed products between
$H_{4}$ and $k[C_{n}]$ namely $H_{4n, \, 1}$, $H_{4n, \, \xi}$ and
$H_{4n, \, \xi^{2}}$ if $\nu (n) = 8$ and respectively $H_{4n, \,
1}$, $H_{4n, \, \xi}$ and $H_{4n, \, \xi^{3}}$ for $\nu(n) = 9$.
\end{example}

We conclude the discussion on these family of quantum groups by
describing the group of Hopf algebra automorphisms of $H_{4n, \,
\xi^{t}}$. For any $t = 0, 1,\cdots, \nu(n) - 1$ we define:
$$
U_{t}(\ZZ_{n}) := \{ \hat{s}  \in U(\ZZ_{n}) \, ; \,  \nu(n) ~|~
t(s-1)\}
$$
$$
V_{t}(\ZZ_{n}) : = \{ \hat{l}  \in U(\ZZ_{n}) \, ; \, 2t(l-1) =
\nu(n)q, \,\, {\rm for} \,\, {\rm some} \,\, {\rm odd} \,\, {\rm
integer} \,\, q\}
$$
$$
\widetilde{U_{t}}(\ZZ_{n}) := U_{t}(\ZZ_{n})\cup V_{t}(\ZZ_{n})
$$

\begin{corollary}\colabel{auto}
Let $k$ be a field of characteristic $\neq 2$, $n$ a positive
integer, $\xi$ a generator of the group $U_n(k)$ of order $\nu(n)$
and $t = 0, 1,\cdots, \nu(n) - 1$. Then $U_{t}(\ZZ_{n})$,
$\widetilde{U_{t}}(\ZZ_{n})$ are subgroups of the group of units
$U(\ZZ_{n})$ and there exists an isomorphism of groups
$$
\Aut _{\rm Hopf}(H_{4n, \, \xi^{t}}) \simeq k^{*}\times
U_{t}(\ZZ_{n}) \qquad {\rm or} \qquad \Aut_{\rm Hopf}(H_{4n, \,
\xi^{t}}) \simeq k^{*} \times \widetilde{U_{t}}(\ZZ_{n})
$$
the latter case holds if and only if $n$ and $\nu(n)$ are both
even.
\end{corollary}

\begin{proof}
To start with we prove that $U_{t}(\ZZ_{n})$ and
$\widetilde{U_{t}}(\ZZ_{n})$ are subgroups of $U(\ZZ_{n})$.
Indeed, take $s$, $p \in U_{t}(\ZZ_{n})$. Then $\nu(n) ~|~
t(s-1)$, $\nu(n) ~|~ t(p-1)$. It follows that $\nu(n) ~|~
t(s-1)(p-1)$ which is equivalent to $\nu(n) ~|~ t(sp - p - s +
1)$. Moreover, we also have $\nu(n) ~|~ t(s + p - 2)$. Thus, we
get $\nu(n) ~|~ t(sp - 1)$. Hence, we obtain $sp \in
U_{t}(\ZZ_{n})$ as desired.

Take now $s \in U_{t}(\ZZ_{n})$ and $l \in V_{t}(\ZZ_{n})$. Then
$2t(l-1) = \nu(n) q$, for an odd integer $q$ and $t(s-1) = \nu(n)
w$, for some integer $w$. It follows that $t(s-1)(l-1) = \nu(n) w
(l-1)$ which is equivalent to $t(sl - s - l + 1) = \nu(n) w
(l-1)$. Moreover, we also have $2t(s + l - 2) = \nu(n) (2w+q)$.
Hence we get $2t(sl-1) = \nu(n) \bigl(2wl + q\bigl)$ and since
$2wl + q$ is odd it follows that $sp \in V_{t}(\ZZ_{n})\subset
\widetilde{U_{t}}(\ZZ_{n})$ as desired. Finally, if $l$, $r \in
V_{t}(\ZZ_{n})$ then $2t(l-1) = \nu(n)q_{1}$ and $2t(r-1) =
\nu(n)q_{2}$, for some odd integers $q_{1}$, $q_{2}$. It is
straightforward to see that $2t(lr-1) = \nu(n) \bigl(q_{1}(r-1) +
(q_{1} + q_{2})\bigl)$ and since $q_{1}(r-1) + (q_{1} + q_{2})$ is
an even integer we get $\nu(n) ~|~ (lr - 1)$ and thus $lr \in
U_{t}(\ZZ_{n}) \subset \widetilde{U_{t}}(\ZZ_{n})$.

Suppose first that $n$ is odd. Then according to the proof of
\thref{2.2}, any Hopf algebra automorphism of $H_{4n, \, \xi^{t}}$
has the following form:
\begin{equation}\eqlabel{class1}
\psi_{\gamma, \, s}: H_{4n, \, \xi^{t}} \to H_{4n, \, \xi^{t}},
\qquad \psi_{\gamma, \, s}(a \# h) = u(a) \# v(h)
\end{equation}
where the Hopf algebra maps $u: H_{4} \rightarrow H_{4}$ and $v:
k[C_{n}] \to k[C_{n}]$ are defined as follows:
$$
\qquad u(1) = 1, \quad u (g) = g, \quad u(x) = \gamma \, x, \quad
v(c) = c^{s}
$$
for some non-zero scalar $\gamma \in k^*$ and $s \in
U_{t}(\ZZ_{n})$ . By a straightforward computation it can be seen
that $\psi_{\gamma, \, s} \circ \psi_{\zeta, \, s'} = \psi_{\gamma
\zeta, \, ss'}$. Then, the following map is an isomorphism of
groups:
$$
\Gamma: \Aut _{\rm Hopf}(H_{4n, \, \xi^{t}}) \to k^{*}\times
U_{t}(\ZZ_{n}), \qquad \Gamma(\psi_{\gamma, \, s}) = (\gamma, \,
s)
$$

Assume now that $n$ is even. Then, again by the proof of
\thref{2.2}, we have two types of Hopf algebra automorphism for
$H_{4n, \, \xi^{t}}$. The first such type of automorphisms is
given by the maps $\psi_{\gamma, \, s}$ defined in \equref{class1}
while the second one is given by:
$$
\varphi_{\sigma, \, l}: H_{4n, \, \xi^{t}} \to H_{4n, \, \xi^{t}},
\qquad \varphi_{\sigma, \, l}(a \# h) = u(a)r(h_{(1)}) \#
v(h_{(2)})
$$
for some $\sigma \in k^*$ and $l \in V_{t}(\ZZ_{n})$, where the
Hopf algebra maps $u: H_{4} \rightarrow H_{4}$, $r: k[C_{n}] \to
H_{4}$ and $v: k[C_{n}] \to k[C_{n}]$ are defined as follows:
$$
\qquad u(1) = 1, \quad u (g) = g, \quad u(x) = \sigma \, x,
\quad r(c) = g, \quad v(c) = c^{l}
$$
Now define $\widetilde{\Gamma}: \Aut_{\rm Hopf}(H_{4n, \,
\xi^{t}}) \to k^{*} \times \widetilde{U_{t}}(\ZZ_{n})$ as follows:
$$
\widetilde{\Gamma}(\psi_{\gamma, \, s}) = (\gamma, \, s)\in k^{*}
\times U_{t}(\ZZ_{n}), \qquad \widetilde{\Gamma}(\varphi_{\sigma,
\, l}) = (\sigma, \, l) \in k^{*} \times V_{t}(\ZZ_{n})
$$
for all $\gamma$, $\sigma \in k^{*}$ and $s \in U_{t}(\ZZ_{n})$,
$l \in V_{t}(\ZZ_{n})$. As $U_{t}(\ZZ_{n}) \cap V_{t}(\ZZ_{n}) =
\emptyset$ it can be easily seen that $\widetilde{\Gamma}$ is well
defined and bijective. Moreover:
$$
\widetilde{\Gamma}(\psi_{\gamma, \, s} \circ \psi_{\zeta, \, s'})
= \widetilde{\Gamma}(\psi_{\gamma \zeta, \, s s'}) = (\gamma
\zeta, \, s s') = (\gamma, \, s)(\zeta, \, s') =
\widetilde{\Gamma}(\psi_{\gamma, \, s})
\widetilde{\Gamma}(\psi_{\zeta, \, s'})
$$
$$
\widetilde{\Gamma}(\psi_{\gamma, \, s} \circ \varphi_{\sigma, \,
l}) = \widetilde{\Gamma}(\varphi_{\gamma \sigma, \, sl}) = (\gamma
\sigma, \, sl) = (\gamma, \, s)(\sigma, \, l) =
\widetilde{\Gamma}(\psi_{\gamma, \,
s})\widetilde{\Gamma}(\varphi_{\sigma, \, l})
$$
$$
\widetilde{\Gamma}(\varphi_{\sigma, \, l} \circ \varphi_{\tau, \,
l'}) = \widetilde{\Gamma}(\psi_{\sigma \tau, \, ll'}) = (\sigma
\tau, \, ll') = (\sigma, \, l)(\tau, \, l') =
\widetilde{\Gamma}(\varphi_{\sigma, \, l})
\widetilde{\Gamma}(\varphi_{\tau, \, l'})
$$
i.e. $\widetilde{\Gamma}$ is a morphism of groups.
\end{proof}

\begin{examples}
$1.$ Let $K = C_2 \times C_2 = \lan a, b \, | \, a^{2} = b^{2} =
1, \, ab = ba \ran$ be the Klein's group and $k$ a field of
characteristic $\neq 2$. By a straightforward computation, similar
to the one performed in \prref{mapex1}, we can prove that $(H_4,
k[C_2\times C_2], \triangleleft, \triangleright)$ is a matched
pair if and only if both actions $\triangleleft$, $\triangleright$
are trivial or the right action $\triangleleft$ is trivial and the
left action $\triangleright = \triangleright^{i}$, for $i = 1, 2,
3$, where $\triangleright^{i}$ are given as follows:
\begin{center}
\begin{tabular} {l | r  r  r  r  }
$\triangleright^{1}$ & 1 & $g$ & $x$ & $gx$\\
\hline 1 & 1 & $g$ & $x$ & $gx$\\
$a$ & 1 & $g$ & $x$ & $gx$ \\
$b$ & 1 & $g$ & $-x$ & $-gx$ \\
$ab$ & 1 & $g$ & $-x$ & $-gx$ \\
\end{tabular} \qquad
\begin{tabular} {l | r  r  r  r  }
$\triangleright^{2}$ & 1 & $g$ & $x$ & $gx$\\
\hline 1 & 1 & $g$ & $x$ & $gx$\\
$a$ & 1 & $g$ & $-x$ & $-gx$ \\
$b$ & 1 & $g$ & $x$ & $gx$ \\
$ab$ & 1 & $g$ & $-x$ & $-gx$ \\
\end{tabular} \qquad
\begin{tabular} {l | r  r  r  r  }
$\triangleright^{3}$ & 1 & $g$ & $x$ & $gx$\\
\hline 1 & 1 & $g$ & $x$ & $gx$\\
$a$ & 1 & $g$ & $-x$ & $-gx$ \\
$b$ & 1 & $g$ & $-x$ & $-gx$ \\
$ab$ & 1 & $g$ & $x$ & $gx$ \\
\end{tabular}
\end{center}
However, using \thref{toatemorf}, we can prove after a rather long
but straightforward computation that all the associated bicrossed
products $H_4 \bowtie k[C_2 \times C_2]$ are isomorphic to the
trivial one, namely the tensor product $H_{4} \ot k[C_2 \times
C_2]$ of Hopf algebras.\footnote{A detailed proof of these results
can be provided upon request.}

$2.$ The examples we have chosen in this section are relevant for
showing that there might be an unexplored theory in finding some
new families of finite quantum groups which are the bicrossed
product of two given finite dimensional Hopf algebras $A$ and $H$.
They can serve as a model for proving similar results in the
future. In \cite{Gabi} the problem was solved for the case $A = H
= H_4$ by showing that a Hopf algebra $E$ factorizes through $H_4$
and $H_4$ if and only if $E \cong H_4 \ot H_4$ or $E \cong H_{16,
\, \lambda}$, for some $\lambda \in k$, where $H_{16, \, \lambda}$
is the Hopf algebra generated by $g$, $x$, $G$, $X$ subject to the
relations:
$$
g^2 = G^2 = 1 \quad  x^2 = X^2 = 0, \quad gx = -xg, \quad GX =
-XG,
$$
$$
gG = Gg, \quad gX = -Xg, \quad x G = - Gx, \quad xX + Xx = \lambda
\, (1 -  Gg)
$$
with the coalgebra structure given by
$$
\Delta (g) = g \ot g, \quad \Delta (x) = x\ot 1 + g \ot x, \quad
\Delta (G) = G \ot G, \quad \Delta (X) = X\ot 1 + G \ot X,
$$
$$
\varepsilon (g) = \varepsilon (G) = 1, \qquad \varepsilon (x) =
\varepsilon (X) = 0
$$
However, up to an isomorphism of Hopf algebras there are only
three Hopf algebras that factorize through $H_4$ and $H_4$, namely
$$
H_4 \ot H_4, \qquad H_{16, \, 0} \qquad {\rm and} \qquad H_{16, \,
1} \cong D(H_4)
$$
where $D(H_4)$ is the Drinfel'd double of $H_4$. For the proofs of
these theorems and related results we refer to \cite{Gabi}.
\end{examples}

\section{Conclusions, outlooks and open problems} \selabel{probnoi}
The bicrossed product for groups or Hopf algebras has served as a
model for similar constructions in other fields of study:
algebras, coalgebras, Lie groups and Lie algebras, locally compact
groups, multiplier Hopf algebras, locally compact quantum groups,
groupoids, von Neumann algebras, etc. Thus, all the results proven
in this paper at the level of Hopf algebras can serve as a model
for obtaining similar results for all the fields mentioned above.
This is the first direction for further study.

The second one is a continuation of the classification problem for
Hopf algebras that factorize through two given Hopf algebras. The
first question in this direction is the following:

\textbf{Question 1:} \emph{Let $G$ and $H$ be two finite cyclic
groups. Describe by generators and relations and classify all
bicrossed products $k[G]^* \bowtie k[H]$ and $k[G]^* \bowtie
k[H]^*$?}

The third direction for further study is given by some open
questions that are directly related to the results of this paper.
For instance, having as a starting point \coref{endoDH} we can
ask:

\textbf{Question 2:} \emph{Let $G$ and $H$ be two finite groups
such that there exists an isomorphism of Hopf algebras $D(k[G])
\cong D(k[H])$. What is the relation between the groups $G$ and
$H$?}

We believe that there exist two non-isomorphic groups such that
their corresponding Drinfel'd doubles are isomorphic but
unfortunately we could not find such an example.

As we already mentioned in the introduction, one of the most
important results on the structure of products of groups is Ito's
theorem \cite{Ito}: any product of two abelian groups is a
meta-abelian group. Starting from here we can ask if an Ito type
theorem holds for Hopf algebras, that is:

\textbf{Question 3 (Ito theorem for Hopf algebras):} \emph{Let
$(A, H, \triangleleft, \triangleright)$ be a matched pair between
two commutative Hopf algebras. Is the bicrossed product $A \bowtie
H$ a meta-abelian Hopf algebra?}

What should a meta-abelian Hopf algebra be is part of the problem.
Having in mind the group case, one of the possible definitions is
the following: a Hopf algebra is called meta-abelian if it is
isomorphic as a Hopf algebra to a crossed product of two
commutative Hopf algebras in the sense of \cite{agorecia}.

Kaplansky's tenth conjecture was invalidated at the end of the
90's.  However, we believe that the following bicrossed version of
it might be true:

\textbf{Question 4 (Bicrossed Kaplansky's tenth conjecture):}
\emph{Let $A$ and $H$ be two finite dimensional Hopf algebras. Is
the set of types of isomorphisms of all bicrossed products $A
\bowtie H$ finite?}

We can iterate the bicrossed product by constructing new examples
of Hopf algebras obtained as a sequence of the form $ ((A
\bowtie^1 H_1) \bowtie^2 H_2) \cdots$. The following is a natural
question in the context:

\textbf{Question 5:} \emph{Let $A$, $H_1$, $H_2$ be Hopf algebras
and $(A, H_1, \triangleleft^1, \triangleright^1)$, $(A\bowtie^1
H_1, H_2, \triangleleft^2, \triangleright^2)$ matched pairs of
Hopf algebras where $A\bowtie^1 H_1$ is the bicrossed product
associated to the first matched pair. Do there exist two matched
pairs of Hopf algebras $(H_1, H_2, \triangleleft^{2'},
\triangleright^{2'})$ and $(A, H_1 \bowtie^{2'} H_2,
\triangleleft^{1'}, \triangleright^{1'})$ such that the canonical
map
$$
(A \bowtie^1 H_1) \bowtie^2 H_2 \to A \bowtie^{1'} (H_1
\bowtie^{2'} H_2), \quad (a \bowtie^1 g) \bowtie^2 h \mapsto a
\bowtie^{1'} (g \bowtie^{2'} h)
$$
is an isomorphism of Hopf algebras? }

\section*{Acknowledgment} The authors would like to thank the referee for his careful reading of this manuscript, as well as Marc Keilberg
for his remarks on \coref{endoDH}.

\end{document}